\documentclass[leqno,12pt]{amsart}
\usepackage{epsfig,color,amssymb}
\usepackage{amssymb,latexsym}
\usepackage{graphicx}
\usepackage{epsfig,color}

\newtheorem{Thm}[equation]{Theorem}
\newtheorem{Cor}[equation]{Corollary}
\newtheorem{Lem}[equation]{Lemma}
\newtheorem{Pro}[equation]{Proposition}

\theoremstyle{definition}

\newtheorem{Def}[equation]{Definition}

\theoremstyle{remark}

\newtheorem{Rem}[equation]{Remark}

\newtheorem{notation}[equation]{Notation}

\numberwithin{equation}{section}
\makeatletter
\renewcommand{\c@figure}{\c@equation}
\makeatother

\newcommand{\T}{{\mathcal T}}

\newcommand{\ra}{\rightarrow}

\newcommand{\bord}{\partial}

\newcommand{\eop}[1]{{\flushright\hfill\fbox{#1}}}

\begin{document} 

\title[A Harmonic Conjugate]{Combinatorial Harmonic Maps and Convergence to Conformal Maps, I: A Harmonic Conjugate.}

\author{Sa'ar Hersonsky}

\address{Department of Mathematics\\ 
University of Georgia\\ 
Athens, GA 30602}

\urladdr{http://www.math.uga.edu/~saarh}
\email{saarh@math.uga.edu}
\date{\today}
\maketitle

\begin{abstract}

In this paper, we provide new discrete uniformization theorems for bounded,  $m$-connected planar domains.  
To this end, we consider a planar, bounded, $m$-connected domain  $\Omega$, and let $\bord\Omega$ be its boundary. Let $\mathcal{T}$ denote a triangulation of $\Omega\cup\bord\Omega$.  We construct a \emph{new} decomposition of $\Omega\cup\bord\Omega$ into a finite union of quadrilaterals with disjoint interiors. The construction is based on utilizing a {\it pair} of harmonic functions on ${\mathcal T}^{(0)}$ and properties of their level curves. 
In the sequel \cite{Her3}, it will be proved that a particular discrete scheme based on these theorems converges to a conformal map,  thus  providing an affirmative answer to a question raised by Stephenson \cite[Section 11]{Steph}. 



\end{abstract}

\maketitle

\section{Introduction}
\label{se:Intro}  

\subsection{Perspective}
\label{se:perspective}
The {\it Uniformization Theorem} for surfaces says that any simply connected Riemann surface is {\it conformally equivalent} to one of three known Riemann surfaces: the open unit disk, the complex plane or the Riemann sphere.  This remarkable theorem is a vast generalization of the celebrated {\it Riemann Mapping Theorem} asserting that 
a non-empty simply connected open subset of the complex plane (which is not the whole of it)   is conformally equivalent to the open unit disk.

\smallskip

Our work in this paper is motivated by the following fundamental questions:

\smallskip

{\it Given a topological surface endowed with some combinatorial data, such as a triangulation, can one use the combinatorics and the topology to obtain an effective version of uniformization theorems, or other types of uniformization theorems?}

\smallskip

 The nature of the input suggests that one should first prove {\it discrete uniformization theorems}, i.e.,  first  provide  a rough approximation to the desired uniformization map and target. Experience shows that this step is not easy to establish, since the input is coarse in nature (such as a triangulation of the domain), and the output should consist of a map from the domain to a surface endowed with some kind of a geometric structure. 
 
 \smallskip
 
Ideally, the approximating maps should have nice properties and if this step is successfully completed, one  then tries to prove convergence of these maps and the output objects attained, under suitable conditions, to concrete geometric objects. 

\smallskip

Let us describe two examples exploiting the usefulness of such an approach (see for instance \cite{Lu} and \cite{Gl} for other important results). A beautiful and classical result which was first proved by Koebe \cite{Koe}, {\it The Discrete Circle Packing Theorem}, states:

\smallskip

{\it Given a finite planar graph (without multiple edges or loops), there exists a packing of Euclidean disks in the plane, enumerated by the vertices of the graph, such that the contact graph of the packing looks exactly like the given graph, that is, the two graphs are isomorphic.}

\smallskip

This theorem was later rediscovered by Thurston  \cite[Chapter 13]{Th2} as a consequence of Andreev's Theorem \cite{An1,An2} concerning hyperbolic polyhedra in terms of circles on the Riemann sphere. Thurston  envisioned  \cite{Th1} a remarkable  application to the theory of conformal mapping of the complex plane and the Riemann sphere. Thurston conjectured that a discrete scheme based on the Discrete Circle Packing Theorem converges to the Riemann mapping. The conjecture which was proved in 1987 by Rodin and Sullivan \cite{RoSu} provides a refreshing geometric view on Riemann's Mapping Theorem. 

\smallskip

Thurston suggested to Schramm to study the case where the sets in the plane that form the tiles in the packing are squares. This resulted in {\it The Finite Riemann Mapping Theorem}  which was proved by Schramm  \cite{Sch} and independently by Cannon, Floyd and Parry \cite{CaFlPa}, in the period 1986--1991:

\smallskip

{\it Let $\T$ be a triangulation of a topological planar quadrilateral. Then there is a tiling of a rectangle by squares, indexed by the vertices of $\T$,  such that the contact graph of the packing looks exactly like the given graph, that is, the two graphs are isomorphic.}

\smallskip

 The problem of tiling a rectangle by squares, as provided by the theorem above, is in some sense a discrete analogue of finding a conformal map from a given quadrilateral to a rectangle (taking corners to corners and boundary to boundary). In this scheme, each vertex is {\it expanded} to a square, the width of the square is a rough estimate to the magnitude of the derivative of the uniformizing analytic map at that vertex. In \cite{CaFlPa} and in \cite{Sch}, it was proved that all the information which is required to get the square tiling above,  is given by a solution of an extremal problem which is a discrete analogue of the notion of {\it extremal length} from complex analysis.   
 
\smallskip 
 
 The actual theorem proved by Cannon, Floyd and Parry (\cite[Theorem 3.0.1]{CaFlPa}) is a bit different and slightly more general than the one stated above. Their solution is also based on discrete extremal length arguments. Another proof of an interesting generalization of \cite{Sch}  was given by Benjamini and Schramm \cite{BeSch1} (see also \cite{BeSch2} for a related study).

\smallskip

 The theme of realizing a given combinatorial object by a packing of concrete geometric objects has a fascinating history which pre-dated Koebe. 
In 1903,  Dehn  \cite{D} showed a relation between square tilings and electrical networks. 
Later on,  in the 1940's, Brooks, Smith, Stone and Tutte explored a foundational correspondence between a square tiling of a rectangle and a planar multigraph with two poles, a source and a sink  \cite{BSST}.  In 1996, Kenyon generalized Dehn's construction and established a correspondence between certain planar non-reversible Markov chains and trapezoid tilings of a rectangle  \cite{Ke}.

\smallskip

In \cite{Her1} and \cite{Her2}, we addressed (using methods that transcend Dehn's idea) the case where the domain has higher connectivity. These papers provide the first step towards an approximation of conformal maps from such domains onto a certain class of flat surfaces with conical singularities.

\subsection{Motivation and the main ideas of this paper}
\label{se:net}
\smallskip

In his attempts to prove uniformization, Riemann suggested considering a planar annulus as made of a uniform conducting metal plate. When one applies voltage to the plate, keeping one boundary component at voltage $k$ and the other at voltage $0$, {\it electrical current}\   will flow through the annulus. The   {\it equipotential}  lines  form a family of disjoint simple closed curves foliating the annulus and separating the boundary curves. The \emph{current} flow lines consist of simple disjoint arcs connecting the boundary components,  and they foliate the annulus as well. Together, the two families provide  ``rectangular" coordinates on the annulus that turn it into a right circular cylinder, or a (conformally equivalent) circular concentric annulus. 

\smallskip

In this paper, we will follow Riemann's perspective on uniformization  by constructing ``rectangular" coordinates from given combinatorial data. The foundational modern theory of boundary value problems on graphs
 enables us to provide a unified framework to the discrete uniformization theorems mentioned above, as well as to more general situations. The important work of Bendito, Carmona and Encinas (see for instance \cite{BeCaEn1},\cite{BeCaEn2} and \cite{BeCaEn3}) is essential for our applications, and 
  parts of it were utilized quite frequently in \cite{Her}, \cite{Her1}, \cite{Her2}, this paper, and its sequel \cite{Her3}.  
    
\smallskip

Consider a planar, bounded, $m$-connected domain $\Omega$, and let $\bord\Omega$ be its boundary which comprises Jordan curves. Henceforth, let $\mathcal{T}$ denote a triangulation of $\Omega\cup\bord\Omega$. 
We will construct a \emph{new} decomposition of $\Omega\cup\bord\Omega$ into ${\mathcal R}$,  a finite union of piecewise-linear quadrilaterals with disjoint interiors. We will show that the set of quadrilaterals can be endowed with a finite measure thought of as a combinatorial analogue of the Euclidean planar area measure. 

\smallskip

Next, we construct a pair $(S_{\Omega},f)$ where $S_{\Omega}$ is a 
special type of a 
genus $0$, {\it singular flat surface}, having $m$ boundary components, which is tiled by rectangles and is endowed with $\mu$,  the canonical area measure induced by the singular flat structure. 
The map $f$ is a {\it homeomorphism} from $(\Omega, \partial\Omega)$ onto $S_{\Omega}$. Furthermore, each quadrilateral is mapped to a single rectangle, and its measure is preserved.   

\smallskip

The proof that $f$ is a homeomorphism, as well as the construction of a measure on the space of quadrilaterals, depends in a crucial way on the existence of  a {\it pair} of harmonic functions on ${\mathcal T}^{(0)}$, and a few properties of their level curves. 
\smallskip

The motivation for this paper is two fold. First, 
recall that in the theorems proved in \cite{Her1} (as well as in \cite{Her2}), the analogous mapping  to $f$ was proved to be an {\it energy-preserving} map (in a discrete sense) from  ${\mathcal T}^{(1)}$ onto a particular singular flat surface.  Hence, it is {\it not} possible to extend that map to a homeomorphism defined on the domain. Furthermore, the natural invariant measure considered there is one-dimensional (being concentrated on edges). So that measure is not the one which we expect to converge, as the triangulations get finer, to the planar Lebesgue measure. 

\smallskip
Second, it is shown in \cite[page 117]{Sch} that if one attempts to use the combinatorics of
the hexagonal lattice, square tilings (as provided by Schramm's method) cannot be used as discrete approximations for the Riemann mapping. There is still much effort by Cannon, Floyd and Parry to provide sufficient conditions under which their method will converge to a conformal map in the cases of an annulus or a quadrilateral.

\smallskip
Thus, the outcome of this paper is the construction of \emph{one} approximating map to a conformal map from $\Omega$.
In \cite{Her3}, which  heavily  relies on our work in this paper,  we will show that a scheme of refining the triangulation, coupled with a particular choice of a conductance function in each step (see Section~\ref{pa:fn} for the definition), leads to convergence of the mappings constructed in each step, to a canonical 
{\it conformal} mapping from the domain onto a particular flat surface with conical singularities. This will, in particular, answer a question raised by Stephenson in 1996 \cite[Section 11]{Steph}.

\subsection{The results in this paper}
\label{se:gmodels}
We now turn to a more detailed description of this paper. In order to ease the notation and to follow the logic of the various constructions,  let us focus  on the case of an annulus.
A {\it slit} in an annulus is a fixed, simple, combinatorial path in ${\mathcal T}^{(1)}$, along which $g$ is monotone increasing which joins the two boundary components (Definition~\ref{de:slit}). 
\smallskip

Let $g$ denote the solution of a discrete Dirichlet boundary value problem defined on ${\mathcal T}^{(0)}$ (see Definition ~\ref{de:dbvp}). We will start by extending $g$ to the interior of the  domain: affinely over edges in ${\mathcal T}^{(1)}$  and over triangles ${\mathcal T}^{(2)}$. We will often abuse notation and will not distinguish between a function defined on ${\mathcal T}^{(0)}$ and its extension over $|{\mathcal T}|$. 

\smallskip

For the applications of this paper and its sequel \cite{Her3}, first in creating ``rectangular" coordinates in a topological sense, and second in \cite{Her3} to prove convergence of the maps constructed in Theorem~\ref{th:annulus} and Theorem~\ref{th:DBVP} to conformal maps, it is necessary to introduce a new function denoted by $h$ on ${\mathcal T}^{(0)}$. This function will  be defined on an annulus minus a slit (i.e, a quadrilateral) and will be called the \emph{harmonic conjugate function}; $h$ is the solution of a particular  Dirichlet-Neumann boundary value problem.

\smallskip

In fact, another function $g^{\ast}$ must first be constructed.  This function will  have the same domain as  $h$ and will be called the \emph{conjugate function} of $g$. It is obtained  by integrating (in a discrete sense) the {\it normal derivative} of $g$ along its level curves (Definition~\ref{de:lapandnorm}). Whereas the normal derivative of $g$ is initially defined only at vertices that belong to  $\partial \Omega$, the simple topological structure of the level curves of $g$ permits the extension of the normal derivative to the interior, and thereafter its integration. These level curves are simple, piecewise-linear, closed curves  that separate the two boundary components and foliate the annulus. Definition~\ref{de:definetheta} will formalize this discussion. 

\smallskip

There is a technical difficulty in this construction (and others appearing in this paper) if a pair of adjacent vertices of ${\mathcal T}^{(0)}$ has the same $g$-values.  
One may generalize the definitions and the appropriate constructions, as one solution. 
For a discussion of this approach and others, see \cite[Section 5]{Ke}. Experimental evidence shows that in the case that  the triangulation is complicated enough such equality rarely happens. Henceforth in this paper, we will {\it assume} that 
no pair of adjacent vertices has the same $g$-values (unless they belong to the same boundary component).


\smallskip

The analysis of the level curves of $g^{\ast}$ is the subject of Proposition~\ref{le:levelcurvesoftheta}. Their interaction with the level curves of $g$ is described in Proposition~\ref{th:topooflevel}.  
The level curves of $g$ form a piecewise-linear analogue of the level curves of the smooth harmonic function 
$u(r,\phi)=r$, and those of $h$ form a piecewise-linear analogue of the level curves of the smooth harmonic function 
$v(r,\phi)=\phi$. 

For any function defined on ${\mathcal T}^{(0)}$, and any $t\in \mathbb{R}$, we let $l_t$ denote the level curve of its affine extension  corresponding to the value $t$.

\begin{Def}[Combinatorial orthogonal filling  pair of functions] 
\label{de:pair} Let $(\Omega,\partial \Omega, {\mathcal T})$ be given, where $\Omega$ is  an annulus minus a slit.
A pair of non-negative functions $\phi$ and $\psi$ defined on ${\mathcal T}^{(0)}$
will be called  \emph{combinatorially orthogonal filling}, if  
for any two level curves $l_\alpha$ and $l_\beta$ of $\phi$ and $\psi$, respectively, one has
\begin{equation}
\label{eq:orth}
|l_\alpha   \cap  l_\beta|=1,
\end{equation}
where $|\cdot |$ denotes the number of intersection points between $l_\alpha$ and $l_\beta$. Furthermore, it is required that each one of the families of level curves  is a non-singular foliation of  $(\Omega,\partial \Omega, {\mathcal T})$. 
\end{Def}

Note that  level curves are computed with respect to the affine extensions of $\phi$ and $\psi$, respectively.

\smallskip

By a \emph{simple} quadrilateral, we will mean a triangulated, closed  topological disk with a choice of four distinct vertices on its boundary.
It follows that a combinatorial orthogonal filling pair of functions induces  a cellular decomposition ${\mathcal R}$ of $\Omega\cup\partial\Omega$ such that  each $2$-cell is a simple quadrilateral, and each $1$-cell is included in a level curve of $\phi$ or of $\psi$.  
Such a decomposition will be called a \emph{rectangular combinatorial net}. 

\smallskip

We now record the essential properties of the pair $\{g,h\}$ in $\Omega$, an annulus minus a slit.

\begin{Thm} 
\label{th:pair is orth}
The pair $\{g, h\}$  is combinatorially  orthogonal filling. 
\end{Thm}



We now turn to stating one of our main discrete uniformization theorems. 
In the course of the proofs of  our main theorems, we will first construct a new decomposition of $\Omega$ into a rectangular net, ${\mathcal R}$, the one induced by $\{g,h\}$; then a model surface which is, when $m>2$,   a singular flat surface tiled by rectangles. Finally, we will construct a map between the domain and the model surface and describe its properties.

\smallskip
Let us start with the fundamental case, an annulus.
Given two positive real numbers $r_1$ and $r_2$, and two angles $\phi_1,\phi_2 \in [0,2\pi)$, the bounded domain in the complex plane whose boundary is determined by the two circles, $u(r,\phi)=r_1$, and $u(r,\phi)=r_2$, and the two radial curves $v(r,\phi)=\phi_1$, and $v(r,\phi)=\phi_2$, will be called an {\it annular shell}. Let $\mu$ denote Lebesgue measure in the plane.  In the statement of the next theorem, the measure $\nu$ which is described  in Definition~\ref{de:combdirichletmeasure} is determined by $g$, $g^{\ast}$ and $h$. The quantity $\mbox{\rm period}(g^{\ast})$  is an invariant of $g^{\ast}$ which encapsulates integration of the normal derivative of $g$ along its level curves (see Definition~\ref{de:period}). 

\smallskip	

Our first discrete uniformization theorem is: 

\begin{Thm}[A discrete Dirichlet problem on an annulus]
\label{th:annulus} 

Let ${\mathcal A}$ be a planar annulus endowed with a triangulation ${\mathcal T}$, and let $\partial {\mathcal A}=E_1\sqcup E_2$.
Let $k$ be a positive constant and let $g$ be the solution of the discrete  Dirichlet boundary value problem defined on $(\mathcal A, \partial {\mathcal A}, {\mathcal T})$ (Definition~\ref{de:dbvp}). 

\smallskip

Let $S_{\mathcal A}$ be the concentric Euclidean annulus with its inner and outer radii satisfying
\begin{equation}
\label{eq:dim}
\{r_1,r_2\}=  
\{1, 2\pi \exp\big( \frac{2\pi}{\mbox{\rm period}(g^{\ast})}\,     k\big)\}.
\end{equation}

Then there exist 

\begin{enumerate} 
\item a tiling $T$ of $S_{\mathcal A}$ by annular shells,
\item a homeomorphism 
$$f:({\mathcal A},\bord{\mathcal A},{\mathcal R})\rightarrow (S_{\mathcal A},\partial S_{\mathcal A},T),
$$
such that 
 $f$ is boundary preserving, it maps each quadrilateral in ${\mathcal R}^{(2)}$ onto a single annular shell in $S_{\mathcal A}$; furthermore,  $f$ preserves the measure of each quadrilateral, i.e., $$\nu(R)=\mu(f(R)),\ \mbox{\rm for all}\  R\in {\mathcal R}^{(2)}.$$ 
\end{enumerate}
\end{Thm}

The dimensions of each annular shell in the tiling are determined by the boundary value problem 
 (in a way that will be described later). In our setting, boundary preserving means that the annular shell associated to a quadrilateral in ${\mathcal R}$ with an edge on $\partial\Omega$ will have  an edge on a corresponding boundary component of $S_{\mathcal A}$. 

\medskip

Our second discrete uniformization theorem is Theorem~\ref{th:DBVP} which provides a geometric mapping and a model for the case $m>2$. The model surface that generalizes the concentric annulus in the previous theorem first appeared in  \cite{Her1}. It is a singular flat, genus $0$, compact surface with $m>2$ boundary components endowed with finitely many conical singularities. Each cone singularity is an integer multiple of $\pi/2$. Such a surface is called a {\it ladder of singular pairs of pants}.
 
\smallskip 
 
In order to prove this theorem, we first construct  a
  topological decomposition of $\Omega$ into simpler components; these are annuli and annuli with one singular boundary component,  for which the previous theorem and a slight generalization of it may be applied. The second step of the proof is geometric.  We show that it is possible to glue the different components which share a common boundary in a {\it length} preserving way. This step entails  a new notion of length which is the subject of  Definition~\ref{de:pairflux}.


\subsection{Organization of the paper}
\label{se:plan}
From  \cite{Her1}, we use the description of the  topological properties of singular level curves of the Dirichlet boundary value problem. The most significant one is a description of the topological structure of the connected components of any singular level curve of the solution.  A study of the topology and geometry of the associated level curves and their complements is carried out in \cite[Section 2]{Her1}. From  \cite{Her2}, we use the description of the topological properties of  level curves of the Dirichlet-Neumann boundary value problem on a quadrilateral. 
A modest familiarity with \cite{Her1, Her2} will be useful for reading this paper. 

\smallskip

For the purpose of making this paper self-contained, 
 a few basic definitions and some notations are recalled in Section~\ref{se:pre}, and results from \cite{Her1,Her2} are quoted as needed. In Section~\ref{se:rectnet}, the first main tool of this paper, a conjugate function to $g$ is defined. In Section~\ref{se:proof}, the second main tool of this paper, a harmonic conjugate function and thereafter a rectangular net, are constructed on an annulus minus a slit. In Section~\ref{se:AnnandQaud}, the cases of an annulus and an annulus with one singular boundary component are treated, respectively, by 
 Theorem~\ref{th:annulus} and Proposition~\ref{th:annulussing}.
  Due to the reasons we mentioned in the paragraph preceding this subsection,  these are foundational for the applications of this paper and of \cite{Her3} as well. Section~\ref{se:highconn} is devoted to the proof of Theorem~\ref{th:DBVP}.

\subsection*{Convention} In this paper, we will assume that a fixed affine structure is imposed on $(\Omega,\partial\Omega,{\mathcal T})$. The existence of such a structure is obtained by using {\it normal coordinates} on  $(\Omega,\partial\Omega,{\mathcal T})$ (see \cite[Theorem 5-7]{Sp}). Since our methods depend on the combinatorics of the triangulation, the actual chosen affine structure is not important.

\subsection*{Acknowledgement} It is a pleasure to thank  Ted Shifrin and Robert Varley for enjoyable and inspiring discussions related to the subject of this paper. We are indebted to Bill Floyd and the referee, for their careful reading, comments, corrections, and questions  leading to improvements on an earlier version of this paper.

\section{Finite networks and boundary value problems}                                      
\label{se:pre}
In this section, we briefly review classical notions from harmonic analysis on graphs through the framework of {\it finite networks}. We then describe a procedure to modify a given boundary problem and ${\mathcal T}$. The reader who is familiar with \cite{Her1} or \cite{Her2} may skip to the next section.  

\subsection{Finite networks} 
\label{pa:fn}
In this paragraph, we will mostly be using the notation of Section 2 in \cite{BeCaEn}. Let $\Gamma=(V,E,c)$  be a planar {\it finite network}; that is,  
  a planar, simple, and
finite connected graph with vertex set $V$ and edge set $E$, where each edge $(x,y)\in E$ is
assigned a {\it conductance} $c(x,y)=c(y,x)>0$. Let ${\mathcal P}({V})$
denote the set of non-negative functions on $V$. Given $F\subset V$, we denote by $F^{c}$ its complement in
$V$.  Set
${\mathcal P}(F)=\{u\in {\mathcal P}(V):S(u)\subset F\}$, where $S(u)=\{ x \in V: u(x)\neq 0 \}$.  The set  $\delta F=\{x\in F^{c}: (x,y)\in E\ {\mbox
{\rm for some}}\ y\in F \}$ is called the {\it vertex boundary} of
$F$. Let ${\bar F}=F\cup \delta F$, and let $\bar E=\{(x,y)\in
E :x\in F\}$.
Let
${\bar \Gamma}(F)=({\bar F},{\bar E},{\bar c})$ be the network
such that ${\bar c}$ is the restriction of $c$ to ${\bar E}$. 
We write $x\sim y$ if $(x,y)\in \bar E$. 

\smallskip

The following operators are discrete analogues of
classical notions in continuous potential theory (see for instance \cite{Fu} and  \cite{ChGrYa}).

\begin{Def}
 \label{de:lapandnorm}  
 Let $u\in {\mathcal P}({\bar  F})$. 
 Then 
 for $x\in F$, the function 
 \begin{equation}
 \label{eq:lap}
 \Delta u(x)=\sum_{y\sim x}c(x,y)\left( u(x)-u(y) \right )
 \end{equation} is called
  the \emph{Laplacian} of $u$ at $x$.  For $x\in \delta(F)$, let $\{y_1,y_2,\ldots,y_m\}\in F$ be its neighbors enumerated clockwise.
 The \emph{normal derivative}  of $u$ at a point
$x\in \delta F$ with respect to a set $F$ is 
\begin{equation}
\label{eq:nor}
\frac{\bord u}{\bord n}(F)(x)= \sum_{y\sim x,\
y\in F}c(x,y)  (u(x)-u(y)).
\end{equation} 
Finally,  $u\in {\mathcal P}({\bar F})$ is called \emph{harmonic} in $F\subset V$ if
$\Delta u(x)=0,$ for all $x\in F$.
\end{Def}

\subsection{Harmonic analysis and boundary value problems on graphs}
\label{pa:bvp}

Consider a planar, bounded, $m$-connected region $\Omega$, and let $\bord\Omega$ be its boundary ($m>1$).  Let $\mathcal{T}$ be a triangulation of $\Omega\cup\bord\Omega$. Let $\bord\Omega=E_1\cup E_2$, where $E_1$ and $E_2$ are disjoint, and $E_1$ is the outermost component of $\bord\Omega$.  Invoke a conductance function ${\mathcal C}$ on ${\mathcal T^{(1)}}$, thus making it a finite network, and use it to define the  Laplacian on ${\mathcal T}^{(0)}$.

\smallskip
\begin{notation}
Henceforth, for any  $F\subset V$  and $g:F\rightarrow \mathbb R$, we let $\int _{v\in F} g(v)$ denote  
$\sum _{v\in F} g(v)$. Similarly,  for any  $X\subset {\bar E}$ and $h:X\rightarrow \mathbb R$, we let $\int _{e\in X} h(e)$ denote  $\sum _{e\in X} h(e)$.
\end{notation}

We need to fix some additional data before describing the discrete boundary value problems  that will be employed in this paper.  To this end,
 let $\{\alpha_1,\ldots, \alpha_l\}$ be a collection of closed disjoint arcs contained in $E_1$, and let 
 $\{\beta_1,\ldots, \beta_s\}$ be a collection of closed disjoint arcs contained in 
$E_2$; let $k$ be a positive constant.

\begin{Def} 
\label{de:dnbvp}

The {\it Discrete  Dirichlet-Neumann Boundary Value Problem}  is determined  by requiring that  
\begin{enumerate}
\item $g({{\mathcal T}^{(0)}\cap{\alpha_i}})=k,\ \mbox{\rm for all}\  i=1,\ldots, l$, and $g({ {\mathcal T}^{(0)}\cap {\beta_j}})=0,\ \mbox{\rm for all}\ j=1\ldots s$,  
 
\smallskip

\item $\dfrac{\bord g}{\bord n}({{\mathcal T}^{(0)}\cap      (E_1\setminus (\alpha_1\cup\ldots\cup \alpha_l))}) =  \dfrac{\bord g}{\bord n}({{\mathcal T}^{(0)}\cap   (E_2\setminus (\beta_1\cup\ldots\cup \beta_s))})=0,\ \mbox{\rm for all}\  $ i=1,\ldots,l$ 
\ \mbox{\rm and}\  j=1,\ldots, s$,

\smallskip

\item $\Delta g=0$ at every interior vertex of ${\mathcal T}^{(0)}$, i.e. $g$ is {\it combinatorially  harmonic}, and

\smallskip

\item $\int_{x\in {\mathcal T}^{(0)}\cap \partial\Omega}\frac{\bord g}{\bord n}(\partial \Omega)(x)=0,$ where ($4$) is a necessary consistent condition.
\end{enumerate}

\end{Def}

\begin{Def} 
\label{de:dbvp}

 The {\it Discrete  Dirichlet Boundary Value Problem}  is determined  by requiring that
\begin{enumerate}
\item $g({{\mathcal T}^{(0)}\cap E_1})=k$, $g({{\mathcal T}^{(0)}\cap E_2})=0$,
\medskip
 and
\item $\Delta g=0$ at every interior vertex of ${\mathcal T}^{(0)}$.
\end{enumerate}
These data will be called  a {\it Dirichlet data} for $\Omega$.

\end{Def}

In the figure below, $E_1$ is depicted by the red curve, $E_2$  is depicted by the three blue curves, and $V$ consists of all the vertices that do not belong to $E_1\cup E_2$.

\begin{figure}[htbp]
\begin{center}
 \scalebox{.45}{ \input{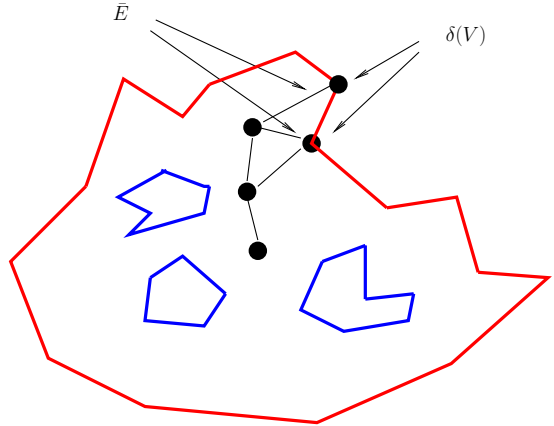}}
 \caption{An example where $V$ consists of all the vertices in the topological interior.}
\label{figure:quad1}
\end{center}
\end{figure}

 \smallskip

A fundamental property which we often will use is the \emph{discrete maximum-minimum principle}, asserting that if $u$ is harmonic on $V'\subset V$, where $V$ is a connected subset of vertices having a connected interior, then $u$ attains its maximum and minimum on the boundary of $V'$ (see \cite[Theorem I.35]{So}).

\smallskip 

The following
proposition (cf. \cite[Prop. 3.1]{BeCaEn}) establishes a discrete version of the first classical {\it Green
identity}. It played an important role in the proofs of the main theorems in  \cite{Her, Her1}, and it  also plays an important role in this paper and in its sequel \cite{Her3}. 

\begin{Pro}[The first Green identity]
\label{pr:Green id} Let $F \subset V$ and $u,v\in {\mathcal P}({\bar
F})$. Then we have that
\begin{equation}
\label{eq:Green}
 \int_{(x,y)\in {\bar
E}}c(x,y)(u(x)-u(y))(v(x)-v(y))=\int_{x\in F}\Delta
u(x)v(x)+\int_{x\in\delta(F)}\frac{\bord u}{\bord n}(F)(x)v(x).
\end{equation}
\end{Pro}

\medskip

\subsection{Piecewise-linear modifications of a boundary value problem} 
\label{pa:simple}
We will often need to modify a given cellular decomposition, and thereafter to modify the initial boundary value problem. The need to do this is twofold. 
First assume, for example, that $L$ is a fixed, simple, closed  level curve  of the initial boundary value problem. 
Since $L\cap {\mathcal T}^{(1)}$ is not (generically) a subset of ${\mathcal T}^{(0)}$,  Definition~\ref{de:pairflux} may not be employed directly  to provide a notion of length to $L$. Therefore, we will add vertices and edges according to  the following procedure. Such new vertices will be called type I vertices.

\smallskip

Let ${\mathcal O}_1, {\mathcal O}_2$ be the two distinct connected components of the complement of $L$ in $\Omega$, with $L$ being the boundary of both (these properties follow by employing the Jordan curve theorem). We will call  ${\mathcal O}_1$ an {\it interior domain} if all the vertices which belong to it have $g$-values that are smaller than the $g$-value of $L$. The other domain will be called the {\it exterior domain}. Note that  by the maximum principle, one of ${\mathcal O}_1, {\mathcal O}_2$ must have all of its vertices with $g$-values smaller than the $g$-value of $L$. 

\smallskip

Let $e\in {\mathcal T}^{(1)}$, and assume that $x=e\cap L$ is a vertex of type I. Thus,  two new edges $(x,v)$ and $(u,x)$ are created. We may assume that $v\in {\mathcal O}_1$ and $u\in {\mathcal O}_2$. Next, define conductance constants $\tilde c(v,x)=\tilde c(x,v)$ and $\tilde c(x,u)=\tilde c(u,x)$ by 

\begin{equation}
\label{eq:prehar}
\tilde c(v,x)= \frac {c(v,u) (g(v)-g(u))}{g(v)-g(x)} \ \   \mbox{\rm and}\ \  \tilde c(u,x) = \frac {c(v,u) (g(u)-g(v))}{g(u)-g(x)}.
\end{equation}

\smallskip

By adding to ${\mathcal T}$ all such new vertices and edges, as well as the piecewise arcs of $L$ determined by the new vertices, we obtain two cellular decompositions, ${\mathcal T}_{{\mathcal O}_1}$ of  
${\mathcal O}_1$ and  ${\mathcal T}_{{\mathcal O}_2}$ of ${\mathcal O}_2$. Note that in general, new two cells that are quadrilaterals are introduced. 

Two conductance functions, ${\mathcal C}_{{\mathcal O}_1}$ and ${\mathcal C}_{{\mathcal O}_2}$, are now defined on the one-skeleton of these cellular decompositions, by modifying according to Equation~(\ref{eq:prehar}) the conductance constants that were used in the Dirichlet data for
$g$ (i.e., changes are occurring only on new edges, and on $L$ the conductance is defined to be identically zero). One then defines (see  \cite[Definition 2.7]{Her1}) a natural modification of the given boundary value problem, the solution of which is easy to control by using the existence and uniqueness theorems in \cite{BeCaEn}. In particular, it is equal to the restriction of 
$g$ to ${\mathcal O}_i$, for $i=1,2$.  

\smallskip

Another technical point which motivates the modification described above will manifest itself in Subsection~\ref{se:fgmetric}. Proposition~\ref{pr:Green id} will be frequently used in this paper, and it may not be directly applied to a modified cellular decomposition, and the modified boundary value problem defined on it. Formally,  in order to apply  Proposition~\ref{pr:Green id} to a meaningful boundary value problem, the modified graph of the network needs to have its vertex boundary components separated enough in terms of the combinatorial distance. Whenever necessary, we will add new vertices along edges and change the conductance constants along new edges in such a way that the solution of the modified boundary value problem will still be harmonic at each new vertex, and will preserve the values of the solution of the initial boundary value problem at the two vertices along the original edge.  Such vertices will be called type II vertices. 

\smallskip

Formally, once such changes occur, a new Dirichlet boundary value problem is defined.
The existence and uniqueness of the solution of a Dirichlet boundary value problem (see \cite{BeCaEn}) allows us to abuse notation and keep denoting the new 
solution by $g$. 
We will also keep denoting by ${\mathcal T}$  any new cellular decomposition obtained as described above.

\section{constructing a conjugate  function on an annulus with a slit}
\label{se:rectnet}
This section has two subsections.  The first subsection contains the construction of the conjugate function $g^{\ast}$  to the solution of the initial Dirichlet boundary value problem defined on an annulus. The second subsection is devoted to the study of the level curve of the conjugate function. In particular, to the interaction between these and the level sets of $g$.

\subsection{Constructing the conjugate function $g^{\ast}$}
\label{se:ortho}

In this subsection, we will construct a function, $g^{\ast}$,  which is {\it conjugate} in a combinatorial sense 
to $g$ (the solution of a Dirichlet boundary value problem defined on an annulus).  The conjugate function will be single valued on the annulus minus a chosen {\it slit}.

\smallskip

Keeping the notation of the previous section and the introduction,
let 
$({\mathcal A}, 
\partial\Omega=E_1\cup E_2,{\mathcal T})$ 
be an annulus endowed with a cellular decomposition in which each $2$-cell is either a triangle or a quadrilateral. Let $k$ be a positive constant, and let $g$ be the solution of a Dirichlet boundary value problem as described in Definition~\ref{de:dbvp}.  Note that all the level curves of $g$ are piecewise, simple, closed curves separating $E_1$ and $E_2$ (see Lemma 2.8 in \cite{Her1} for the analysis in this case and the  case of higher connectivity) which foliate ${\mathcal A}$. 

\smallskip

Before providing the  definition of the conjugate function, we need to make a choice of a  piecewise linear path in  ${\mathcal A}$.

\begin{Def}
\label{de:slit}
Let $\mbox{\rm slit} ({\mathcal A})$ denote a fixed, simple, combinatorial path in ${\mathcal T}^{(1)}$ which joins $E_1$ to $E_2$. Furthermore, we require that the restriction of the solution of the discrete Dirichlet boundary value problem to it  is monotone decreasing.
\end{Def}

\begin{Rem}
The existence of such a path  is guaranteed by the discrete maximum principle.
\end{Rem}

Let 
\begin{equation}
\label{eq:levelsets}
{\mathcal L}=\{L(v_0),\ldots,L(v_k)\}
\end{equation} 
be the collection of  level curves of $g$ that contain all the vertices in ${\mathcal T}^{(0)}$ arranged according to increasing values of $g$. It follows from  Definition~\ref{de:dbvp} that 
$L(v_0)=E_2$ and $L(v_k)=E_1$.  We also add vertices of Type II so that any two level sets in ${\mathcal L}$ are at (combinatorial) distance equal to two. This can be done in various ways, henceforth, we will assume that one of these is chosen.

\smallskip

We wish to construct a single valued function on ${\mathcal A}$. In order to do so, we will start with a preliminary case.  To this end, let ${\mathcal Q}_{\mbox\small{\rm slit}}$ denote the quadrilateral obtained by cutting open ${\mathcal A}$ 
along $\mbox{\rm slit} ({\mathcal A})$ and having two copies of $\mbox{\rm slit} ({\mathcal A})$ attached, keeping the conductance constants along the split edges. 
Since an ${\mathcal A}$ orientation is well defined, we will denote one of the two copies by  $\partial{\mathcal Q}_{\mbox\small{\rm base}}$, and the other by $\partial{\mathcal Q}_{\mbox\small{\rm top}}$. In other words, from the point of view of ${\mathcal A}$, points on $\mbox{\rm slit} ({\mathcal A})$ may be endowed with two labels, recording whether they are the starting point of a level curve (with winding number equal to one) or its endpoint. We keep the values of $g$ at the vertices  unchanged.  Thus, corresponding vertices in $\partial{\mathcal Q}_{\mbox\small{\rm base}}$ and $\partial{\mathcal Q}_{\mbox\small{\rm top}}$ have identical  $g$-values. By abuse of notation, we will keep denoting by ${\mathcal T}^{(0)}$  the $0$-skeleton of ${\mathcal Q}_{\mbox\small{\rm slit}}$.


\begin{figure}[htbp]
\begin{center}
 \scalebox{.50}{ \input{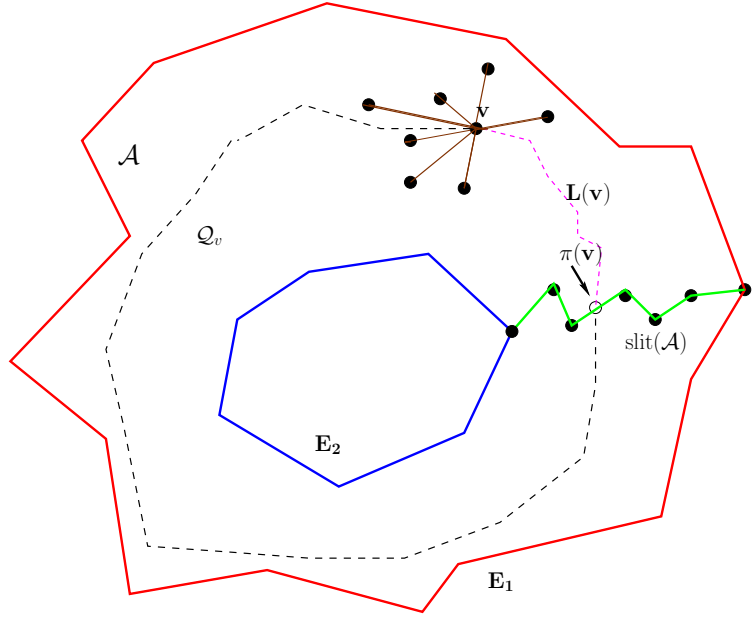}}
 \caption{An example of a quadrilateral $Q_v$.}
\label{figure:quad}
\end{center}
\end{figure}

\smallskip

For $v\in {\mathcal A}\setminus E_2$, which is in ${\mathcal T}^{(0)}$ or a vertex of type I, let $L(v)$ denote the unique level curve of $g$ which contains $v$. Let 
${\mathcal Q}_{v}$ be the interior of the piecewise-linear quadrilateral whose boundary is defined by 
$\partial{\mathcal Q}_{\mbox\small{\rm base}},\partial{\mathcal Q}_{\mbox\small{\rm top}} , L(v)$  and $E_2$.  
For $v\in E_2$, which is in ${\mathcal T}^{(0)}$ or a vertex of type I, 
$\bar{{\mathcal Q}}_{v}$ is defined to be (the interior of) ${\mathcal Q}_{\mbox\small{\rm slit}}$.





\begin{Rem}
Recall that  a vertex of type I is introduced whenever the intersection between an edge 
and the level curve does not belong to ${\mathcal T}^{(0)}$.
\end{Rem}
\smallskip

We now make

\begin{Def}[The Conjugate function of $g$]
\label{de:definetheta}
Let $v$ be a vertex in ${\mathcal T}^{(0)}$ or a vertex of type I. Let
\begin{equation}
\label{eq:starting}
\pi(v)=L(v)\cap \partial{\mathcal Q}_{\mbox\small{\rm base}}. 
\end{equation}

\medskip
We define $g^{\ast}(v)$, the conjugate function of $g$, as follows.

\begin{description}
\item[First case] Suppose that $v\not\in E_2 \cup \partial{\mathcal Q}_{\mbox\small{\rm top}}$. Then 
\begin{equation}
\label{eq:orthfun}
g^{\ast}(v) = \int_{\pi(v)} ^{v}\dfrac{\partial g}{\partial n}({\mathcal Q}_{v})(u),
\end{equation}
where the integration is carried along (the vertices of) $L(v)$ in the counter-clockwise direction.
\medskip

\item[Second case] Suppose that $v\in E_2 \setminus \partial{\mathcal Q}_{\mbox\small{\rm top}}$. Then we define $g^{\ast}(v)$ by
\begin{equation}
\label{eq:orthfun1}
g^{\ast}(v) = \int _{\pi(v)} ^{v}-\dfrac{\partial g}{\partial n}({ \bar {\mathcal Q}}_{v})(u)    = \int_{\pi(v)} ^{v}\big|\dfrac{\partial g}{\partial n}({ \bar {\mathcal Q}}_{v})(u)\big|,
\end{equation}
\end{description}
where the integration is carried along (the vertices of) $E_2$ in the counter-clockwise direction.

\smallskip
On edges in $\partial{\mathcal Q}_{\mbox\small{\rm top}}$, we record the conductance constants induced by ${\mathcal A}$. In order to define $g^{*}$ on $\partial{\mathcal Q}_{\mbox\small{\rm top}}$, we consider the vertices on 
$\partial{\mathcal Q}_{\mbox\small{\rm top}}$ as vertices in ${\mathcal A}$.
For the single vertex $\partial{\mathcal Q}_{\mbox\small{\rm top}}\cap E_2$, the integration above is modified to include the contribution of  its normal derivative from its rightmost neighbor in $\partial{\mathcal Q}_{\mbox\small{\rm top}}$. 
For  any other vertex in $\partial{\mathcal Q}_{\mbox\small{\rm top}}$,  the integration above is modified to include the contribution of its normal derivative from its leftmost neighbor in $\partial{\mathcal Q}_{\mbox\small{\rm top}}$.
 Finally, for a point $z\in {\mathcal Q}_{\mbox\small{\rm slit}}$ which is not a vertex,  $g^{\ast}(z)$ is defined by 
extending $g^{\ast}$ affinely 
 over edges and triangles, and bi-linearly over quadrilaterals.
\end{Def}

\smallskip

\begin{Rem}
The absolute value of the normal derivative of $g$ at a vertex which appears 
in Equation~(\ref{eq:orthfun1}), is due to the maximum principle. The continuity of $g^{\ast}$ from the right on $E_2$ follows from similar arguments to those appearing in the proof of Proposition~\ref{pr:topisalevelcurve} below.
\end{Rem}

\begin{Rem}
Henceforth, we  will denote by $Q(E_1)$, and by $Q(E_2)$, the two boundary components of ${\mathcal Q}_{\mbox\small{\rm slit}}$, which correspond to their counterparts $E_1$, and $E_2$, respectively,  in ${\mathcal A}$. 
\end{Rem}

We now turn to studying topological properties of the level curves of $g^{\ast}$. 

\bigskip

By definition, $\partial{\mathcal Q}_{\mbox\small{\rm base}}$ is the level curve of $g^{\ast}$ which corresponds to $g^{\ast}=0$. We will prove that  $\partial{\mathcal Q}_{\mbox\small{\rm top}}$ is also a level curve of $g^{\ast}$. In other words, computing the value of $g^{\ast}$ at the endpoint of a level curve emanating from $\partial{\mathcal Q}_{\mbox\small{\rm base}}$  is independent of the level curve chosen. The proof is an application of the first Green identity (see Proposition~\ref{pr:Green id}).

\begin{Pro}
\label{pr:topisalevelcurve}
The curve $\partial{\mathcal Q}_{\mbox\small{\rm top}}$ is a level curve of $g^{\ast}$.
\end{Pro}

\begin{proof}
Let $L_1$ and $L_2$ be any two level curves of $g$ which start at $\partial{\mathcal Q}_{\mbox\small{\rm base}}$ and have their endpoints  $x_1$, and $x_2$, on $\partial{\mathcal Q}_{\mbox\small{\rm top}}$, respectively. Let ${\mathcal A}_{(L_1,L_2)}$ denote the (interior of) the annulus whose boundary components are $L_1$, and $L_2$, respectively. Without loss of generality, assume that the $g$-value of $L_1$ is bigger than the $g$-value of $L_2$.  We must show that
\begin{equation}
\label{eq:fluxfortheta}
g^{\ast}(x_1)=g^{\ast}(x_2).
\end{equation}

We now add vertices of type I and II according to the procedure defined in Subsection~\ref{pa:simple}, so that the first Green identity, Proposition~\ref{pr:Green id}, may be applied to a Dirichlet boundary value problem on the network induced on ${\mathcal A}_{(L_1,L_2)}$.

Let $w \equiv 1$ be the constant function defined in ${\mathcal A}_{(L_1,L_2)}$. The assertion of Proposition~\ref{pr:Green id}, applied with the functions $w$ and $g$ on the induced network in ${\mathcal A}_{(L_1,L_2)}$, yields

\begin{equation}
\label{eq:fluxfortheta0}
\int_{x\in   {\mathcal T}^{(0)}\cap   \partial  {\mathcal A}_{(L_1,L_2)}}\frac{\bord g}{\bord n}({\mathcal A}_{(L_1,L_2)})(x)=0.
\end{equation}

Hence, it follows that

\begin{equation}
\label{eq:fluxfortheta00}
\int_{x\in {\mathcal T}^{(0)}\cap    L_2}\frac{\bord g}{\bord n}({\mathcal A}_{(L_1,L_2)})(x) +  \int_{y\in   {\mathcal T}^{(0)}\cap  L_1}\frac{\bord g}{\bord n}({\mathcal A}_{(L_1,L_2)})(y) =0.
\end{equation}
(Note that vertices of type I appear in both of the integrals, so one must apply Equation~(\ref{eq:prehar}) and the discussion preceding it to justify this equality.)
It follows from Definition~\ref{de:definetheta} that the second term in the above equations is equal to 
$g^{*}(x_1)$. Furthermore, since $g$ is harmonic in $ {\mathcal T}^{(0)}\cap   {\mathcal A}$, and since $L_2$ is a level curve of $g$, it follows  that 

\begin{equation}
\label{eq:fluxfortheta1}
\int_{x\in  {\mathcal T}^{(0)}\cap   L_2}\frac{\bord g}{\bord n}({\mathcal A}_{(L_1,L_2)})(x) +   \int_{x\in  {\mathcal T}^{(0)}\cap L_2}\frac{\bord g}{\bord n}({\mathcal A}_{(E_2,L_2)})(x) = 0.
\end{equation}

As above, it follows from Definition~\ref{de:definetheta} that the second term in the above equations is equal to 
$-g^{*}(x_2)$.
Therefore, Equations~(\ref{eq:fluxfortheta00}) and (\ref{eq:fluxfortheta1}) imply that 
\begin{equation}
\label{eq:fluxfortheta3}
g^{*}(x_1)= g^{*}(x_2).
\end{equation}
 This ends the proof of the Proposition. 
 \end{proof}

\begin{Rem}
With easy modifications, the proof goes through when $L_2=E_2$.
\end{Rem}

We now make

\begin{Def}
\label{de:period}
The \emph{period} of $g^{\ast}$ is  defined to be the  $g^{*}$ value on $\partial{\mathcal Q}_{\mbox\small{\rm top}}$, that is, 

\begin{equation}
\label{eq:periodoftheta}
\mbox{\rm period}(g^{\ast})=  g^{*}(\partial{\mathcal Q}_{\mbox\small{\rm top}}\cap {\mathcal T}^{(0)})
  =\int_{u\in {\mathcal T}^{(0)}\cap  E_1}\dfrac{\partial g}{\partial n}({\mathcal A}_{(E_2, E_1)})(u).
\end{equation}
\end{Def}

\medskip
Following similar arguments to these in the proof above, it is easy to check that $\mbox{\rm period}(g^{\ast})$ is independent of the choice of the added vertices of type II. Also, 
note that  $\mbox{\rm period}(g^{\ast})$  is independent of the choice of the level curve chosen or the slit chosen.  
Indeed,  it follows from  Proposition~\ref{pr:topisalevelcurve} that for a {\it fixed} slit the computation of the period is independent of the points chosen on the slit. 

\smallskip
Assume now that a different slit is chosen, and  let $\eta$ be the conjugate function corresponding to the new slit. It readily follows that $\mbox{\rm period}(\eta) = \mbox{\rm period}(g^{\ast})$.

\medskip

Indeed, start with any point $x$ on any of the two slits, let $l_x$, the (unique) level curve of $g$ passing through $x$. The computation of both periods is done by summing the normal derivative of $g$ along (the whole of) $l_x$, hence, they are equal. In fact, their common value is  the integral of the normal derivative of g along $E_2$ (unless $E_1$ is  chosen so an absolute value needs to be applied to the end result).

\bigskip
We now continue the study of the level curves of $g^{\ast}$.   Note that by the maximum principle (applied to $g$), and its definition, $g^{\ast}$ is monotone {\emph strictly increasing} along level curves of $g$. This property will now be  used in the following

\begin{Pro}
\label{le:levelcurvesoftheta}
Each level curve of $g^{\ast}$ has no endpoint in the interior  of  ${\mathcal Q}_{\mbox\small{\rm slit}}$, is simple, and joins $Q(E_1)$ to $Q(E_2)$. Furthermore, any two level curves of $g^{\ast}$ are disjoint.
\end{Pro}

\begin{proof}
Suppose that a level curve of $g^{\ast}$ which starts at $s\in Q(E_2)$ has an endpoint $\xi$ in 
$T\in {\mathcal T}^{(2)}$, where $T$ lies in the interior of 
${\mathcal A}$. Let $[s,\xi]$ be the intersection of this level curve with the interior of ${\mathcal A}$. Let $L_{\xi}$ denote the level curve of $g$ that passes through $\xi$.  Since the level curves of $g$ foliate ${\mathcal A}$, there exists a level curve  $L_\psi$ of $g$, which is as close as we wish to $L_{\xi}$, and such that its intersection with $[s,\xi]$ is empty.
Since $g^{\ast}$ is monotone increasing and continuous along $L_\psi$, it assumes all values between $0$ 
and $\mbox{\rm period}(g^{\ast})$. Hence, it will assume the value $g^{\ast}(\xi)$. This shows that no level curve of $g^{\ast}$ can have an interior endpoint.

Assume that one of the level curves of $g^{\ast}$ is not simple. Let ${\mathcal D}$ be any domain which is bounded by it. Since the level curves of $g$ foliate the annulus, one of these intersects the boundary of ${\mathcal D}$ in at least two points. The monotonicity of $g^{\ast}$ along the  level curves of $g$ renders this impossible.

Assume that  there exists a level curve of $g^{\ast}$, $L(g^{\ast})$, which does not join $Q(E_1)$ to $Q(E_2)$. By construction, each level curve of $g^{\ast}$ does not have an endpoint inside ${\mathcal A}$ and its intersection with each $2$-cell is a segment (or a point). Hence, both endpoints of  $L(g^{\ast})$ must lie on $Q(E_1)$ or on $Q(E_2)$.  Without loss of generality, assume that both endpoints are on $Q(E_1)$.   Hence, there must be a level curve of $g$ that intersects  $L(g^{\ast})$ in at least two points. Reasoning in a similar way to the paragraph above, this easily leads to a contradiction.

The fact that level curves of $g^{\ast}$ that correspond to the same value may not intersect each other follows from similar arguments to those appearing in the first parts of the proof. 

\end{proof}

\smallskip
Of special importance is the interaction between the level curves of $g^{\ast}$ and the level curves of $g$. The following proposition states that,  from a topological point of view, the union of the two families of level curves resembles  a planar coordinate system. 
This proposition is one  topological prerequisite for the proof of Theorem~\ref{th:pair is orth},  which will appear in the next subsection. 

\begin{Pro}
\label{th:topooflevel}
The number of intersections between any level curve of $g^{\ast}$ and any level curve of $g$ is equal to $1$. 
\end{Pro}

\begin{proof}
It readily follows from the proof of Proposition~\ref{le:levelcurvesoftheta} that the number of intersections of any level curve of $g$ with any level curve of $g^{\ast}$ is at most equal to one. Since both families of level curves foliate 
${\mathcal Q}_{\mbox\small{\rm slit}}$, this number is equal
to one.

\end{proof}


\section{Constructing a harmonic conjugate and the proof of theorem~\ref{th:pair is orth}}
\label{se:proof} 

This section has three subsections. 
In the first, we define the {\it harmonic conjugate function} $h$ and study its immediate properties. In the second,  we provide the proof of Theorem~\ref{th:pair is orth}.  In the third, we  define the pair-flux metric and its induced length. These notions will be essential to the proof of Theorem~\ref{th:DBVP}  in which gluing two components of the complement of a singular level curve  of the solution takes place.

\subsection{A harmonic conjugate function}
We keep the notation of the previous section and modify Definition~\ref{de:dnbvp} to the  case of ${\mathcal Q}_{\mbox\small{\rm slit}}$.

 \begin{Def}[]
 \label{de:harconj}
The harmonic conjugate function ${h}$ is the solution of the discrete Dirichlet-Neumann boundary value problem defined by 
\begin{enumerate}
\item $h({{\mathcal T}^{(0)}\cap\partial{\mathcal Q}_{\mbox\small{\rm top}}})= \mbox{\rm period}(g^{\ast})$, \mbox{\rm and}\   \ $h({ {\mathcal T}^{(0)}\cap\partial{\mathcal Q}_{\mbox\small{\rm base}}})=0,$  
 
\medskip

\item $\dfrac{\bord {h}}{\bord n}({{\mathcal T}^{(0)}\cap Q(E_1)}) =  \dfrac{\bord {h}}{\bord n}({{\mathcal T}^{(0)}\cap Q(E_2)})=0$\  (\mbox{\rm other than at the four corners of}\ ${\mathcal Q}_{\mbox\small{\rm slit}}$),  

\medskip

\item $\Delta {h}=0\ \mbox{\rm at every (interior) vertex of ${\mathcal T}^{(0)}\cap {\mathcal Q}_{\mbox\small{\rm slit}}$, and }$

\medskip

\item $\int_{x\in {\mathcal T}^{(0)}\cap \partial{ \mathcal Q}_{\mbox\small{\rm slit}} }\frac{\bord {h}}{\bord n}(\partial \Omega)(x)=0,\ \mbox{\rm  where ($4$) is a necessary consistent condition}.$
\end{enumerate}
\end{Def}

\medskip

Consider now 
\begin{equation}
\label{eq:levelsetsh}
{\mathcal M}=\{M(v_0),\ldots,M(v_p)\},
\end{equation} 
 the collection of  level curves of $h$, that contain all the vertices in ${\mathcal T}^{(0)}$ arranged according to increasing values of $h$. It follows from the definition of $h$ that 
$M(v_0)= {\mathcal Q}_{\mbox\small{\rm base}} $ and $M(v_p)= {\mathcal Q}_{\mbox\small{\rm top}}$. 

\smallskip

We will now define the {\it conjugate function} of ${h}$, which will be denoted by $h^{\ast}$, to the case of the quadrilateral  
${\mathcal Q}_{\mbox\small{\rm slit}}$;  it is a straightforward modification of the definition of $g^{\ast}$ (Definition~\ref{de:definetheta}). 

\smallskip

Indeed, one recalls that by  \cite[Proposition 2.1]{Her2}  the level curves of $h$ are disjoint, piecewise-linear simple curves that foliate ${\mathcal Q}_{\mbox\small{\rm slit}}$ and join  $Q(E_1)$  to $Q(E_2)$. 

\smallskip

For $v\in {\mathcal Q}_{\mbox\small{\rm slit}} \setminus  {\mathcal Q}_{\mbox\small{\rm base}} $, which is in ${\mathcal T}^{(0)}$ or a vertex of type I, let $M(v)$ denote the unique level curve of $h$ which contains $v$. Let 
${\mathcal P}_{v}$ be the piecewise-linear quadrilateral whose boundary is defined by 
$Q(E_1), M(v), Q(E_2)$ and ${\mathcal Q}_{\mbox\small{\rm base}}$.
For $v\in  {\mathcal Q}_{\mbox\small{\rm base}}$,  recall that  ${\mathcal Q}_{\mbox\small{\rm base}} =M(v_0)$ is the unique level curve of $h$ which contains $v$. Let 
$\bar{{\mathcal P}}_{v}$ be equal to ${\mathcal Q}_{\mbox\small{\rm slit}}$.


\begin{Rem}
Note that  a vertex of type I is introduced whenever the intersection between an edge 
and the level curve does not belong to ${\mathcal T}^{(0)}$.
\end{Rem}

\begin{Def}[The conjugate function of $h$]
\label{de:conjofh}

Let $v$ be a vertex in ${\mathcal T}^{(0)}\cap {\mathcal Q}_{\mbox\small{\rm slit}}$ or a vertex of type I. Let
\begin{equation}
\label{eq:starting1}
\pi(v)=M(v)\cap Q(E_1). 
\end{equation}

\medskip
We define $h^{\ast}(v)$, the conjugate function of $h$, as follows.

\begin{description}
\item[First case] Suppose that $v\not\in {\mathcal Q}_{\mbox\small{\rm base}}$. Then 
\begin{equation}
\label{eq:orthfunh}
h^{\ast}(v) = \int_{\pi(v)} ^{v}\dfrac{\partial h}{\partial n}({\mathcal P}_{v})(u),
\end{equation}
where the integration is carried along (the vertices of) $M(v)$ (from $\pi(v)$ to $v$).
\medskip

\item[Second case] Suppose that $v\in   {\mathcal Q}_{\mbox\small{\rm base}} $. Then we define $h^{\ast}(v)$ by
\begin{equation}
\label{eq:orthfunh1}
h^{\ast}(v) = \int_{\pi(v)} ^{v}\big|\dfrac{\partial h}{\partial n}({ \bar {\mathcal P}}_{v})(u)\big|,
\end{equation}
\end{description}
where the integration is carried along (the vertices of) ${\mathcal Q}_{\mbox\small{\rm base}}$ (from $\pi(v)$ to $v$).

\smallskip

For a point $z\in {\mathcal Q}_{\mbox\small{\rm slit}}$, which is not a vertex as above,  $h^{\ast}(z)$ is defined by 
extending $h^{\ast}$ affinely 
 over edges and triangles, and bi-linearly over quadrilaterals.
\end{Def}

\medskip

We now turn to studying a few topological properties of the level curves of $h^{\ast}$ and their interaction with the level curves of $h$. The statements and the proofs are immediate generalizations of their counterparts in Section~\ref{se:rectnet}, and therefore we omit the proofs. The interaction between the level curves of $g$ and those of $h$ is subtle and will be treated in the next subsection. 

\medskip

By definition, $Q(E_1)$ is the level curve of $h^{\ast}$ which corresponds to $h^{\ast}=0$. It will follow that   $Q(E_2)$ is also a level curve of $h^{\ast}$. In other words, computing the value of $h^{\ast}$ at the endpoint of a level curve emanating from $Q(E_1)$  is independent of the level curve chosen.  We recall this property in

\begin{Pro}
\label{pr:rightisalevelcurve}
The curve $Q(E_2)$ is a level curve of $h^{\ast}$ in ${\mathcal Q}_{\mbox\small{\rm slit}}$.
\end{Pro}
The proof is an application of the first Green identity (see Proposition~\ref{pr:Green id}) and is a direct generalization of the method of proof of Proposition~\ref{pr:topisalevelcurve} applied to  ${\mathcal P}_{v}$  where $v\in E_2$.
Although not used in this paper, as a consequence of this proposition, we can now make

\begin{Def}
\label{de:width}
The \emph{width} of $h^{\ast}$ is  defined to be the $h^{*}$ value on $Q(E_2)$, that is,

\begin{equation}
\label{eq:periodofhast}
\mbox{\rm width}(h^{\ast})= h^{\ast}(Q(E_2)\cap {\mathcal T}^{(0)}).   
\end{equation}
\end{Def}

\bigskip

Note that by the maximum principle (applied to $h$), and by its definition, $h^{\ast}$ is monotone \emph{strictly} increasing along level curves of $h$. This property is used in  proving the following proposition in exactly the same way that the analogous property for the pair $\{g, g^{\ast}\}$ was used 
in the proof of Proposition~\ref{le:levelcurvesoftheta}.

\begin{Pro}
\label{le:levelcurvesofhast}
Each level curve of $h^{\ast}$ has no endpoint in the interior  of  ${\mathcal Q}_{\mbox\small{\rm slit}}$, is simple, and joins ${\mathcal Q}_{\mbox\small{\rm base}}$ to $ {\mathcal Q}_{\mbox\small{\rm top}} $. Furthermore, any two level curves of $h^{\ast}$ are disjoint.
\end{Pro}

\smallskip

Of special importance is the interaction between the level curves of $h^{\ast}$ and the level curves of $h$. The following proposition will show that, from a topological point of view, the union of the two families of level curves of $\{h,h^{\ast}\}$ resembles a planar coordinate system. 
This proposition is another topological prerequisite for the proof of Theorem~\ref{th:pair is orth},  which will appear in the next subsection. 

\begin{Pro}
\label{th:topooflevelofhasth}
The number of intersections between any level curve of $h^{\ast}$ and any level curve of $h$ is equal to $1$. 
\end{Pro}

The proof is an immediate modification of the proof of Propostion~\ref{th:topooflevel} to the case of the pair $\{h,h^{\ast}\}$.

\subsubsection{\bf Viewing $h$ from a PDE perspective}
\label{se:conjugate}
The term ``harmonic conjugate" associated with $h$ is motivated by the first three properties used to define $h$ (Definition~\ref{de:harconj}). 
Hence, $h$ satisfies the combinatorial analogues of the analytical properties of
the polar angle function $v(r,\phi)=\phi$ in the complex plane, which is known to be, when it is single-value defined, the harmonic conjugate function of $v(r,\phi)=r$.

\medskip

\subsubsection{\bf Related work.}
Our definition of the harmonic conjugate function is motivated by the fact that, in the smooth category, a conformal map is determined by its real and imaginary parts, which are known to be harmonic conjugates. The search for discrete approximation of conformal maps has a long and rich history. We refer to \cite{Mer} and \cite[Section 2]{ChSm} as excellent recent accounts. 



\smallskip

We should also mention that a search for a combinatorial Hodge star operator has recently gained much attention and is closely related to the construction of a harmonic conjugate function. We refer the reader to  \cite{Hi} and to \cite{Po} for further details and examples for such combinatorial operators.

\subsection{The proof of Theorem~\ref{th:pair is orth}}
\label{se:proofofTheorem0.3}

Each vertex in
${\mathcal T}^{(0)}$ (which is  now a modification of the original one by adding all the   vertices  of type I in ${\mathcal Q}_{\mbox\small{\rm slit}}$) belongs to one and only one of the level curves of $h$.
Let 
\begin{equation}
\label{de:horizontal}
{\mathcal M}=\{M(v_0),\ldots,M(v_p)\}
\end{equation}
be defined according to Equation~(\ref{eq:levelsetsh}); this is  
 the set of level curves of $h$,  arranged according to increasing values of $h$, which contain all of the vertices mentioned above. Recall  that 
$M(v_0)= {\mathcal Q}_{\mbox\small{\rm base}} $ and $M(v_p)= {\mathcal Q}_{\mbox\small{\rm top}}$. Let 
 \begin{equation}
 \label{eq:radial1}
 {\mathcal L}=\{L(v_0),\ldots,L(v_k)\}
 \end{equation}
be defined according to Equation~(\ref{eq:levelsets}); this is  
 the set of level curves of $g$, arranged according to increasing values of $g$, which contain all of the vertices mentioned above. Recall that 
$Q(E_2)\subset L(v_0)$ and $Q(E_1)\subset L(v_k)$. 

\medskip

We will now study the following decomposition of $\Omega\cup\partial\Omega$.  

\begin{Def}
\label{de:Rectangular Net}
Let ${\mathcal R}$ be the decomposition of $\Omega\cup\partial\Omega$ induced by the  intersection of the sets 
$\{ {\mathcal M}, {\mathcal L}  \}$. 
\end{Def}

Since each one of  the sets of level curves of $g$, $h$, respectively, is clearly dense in $\Omega\cup\partial\Omega$, in order to prove Theorem~\ref{th:pair is orth} it suffices to establish 

\begin{Pro}[A rectangular net]
\label{th:netisrecyangular}
The number of intersections between any level curve of $g$ and any level curve of $h$ is equal to $1$. 
\end{Pro}

Once this proof is furnished, it will follow that the each $2$-cell in ${\mathcal R}$ is a quadrilateral, where each pair of opposite boundaries is contained in successive level sets of $h$ or in successive level sets of $g$. Note that a vertex is formed in ${\mathcal R}^{(0)}$, whenever a level set of $g$ and a level set of $h$ intersect.

\bigskip

\noindent\emph{Proof of Proposition~\ref{th:netisrecyangular}.} We argue by contradiction.
It follows from \cite[Lemma 2.8]{Her1} and \cite[Proposition 2.1]{Her2} that the level curves of $g$ as well as the level curves of $h$ foliate ${\mathcal Q}_{\mbox\small{\rm slit}}$; hence, the number of intersections between a level curve of $g$ and a level curve of $h$ is at least one. 

\medskip

By Proposition~\ref{th:topooflevelofhasth}, the number of intersections between any level curve of $h$ and any level curve of $h^{\ast}$ is exactly one. Hence, the proof of the Proposition will readily follow from the following lemma, where the level curves of $g^{\ast}$ and $h^{\ast}$ play an important role.
\begin{Lem}
Suppose that a $g$-level curve $L(v_i)$ intersects an $h$-level curve $M(v_j)$ in at least two points. Then there exists a level curve of  $h^{\ast}$ which intersects $M(v_j)$ in  at least two points. 
\end{Lem}

\begin{proof}
There are several cases to consider. 
First, assume that $u,v$ are the first two intersection points of  $L(v_i)$ and $M(v_j)$, arranged by their increasing $g$-values. Let $ L([u,v]) \subset L(v_i)$ be the arc (on $L(v_i)$) which joins $u$ to $v$, and let $M([u,v])\subset M(v_j)$ be the arc (on  $M(v_j)$) which joins them. Further assume that these arcs are disjoint other than their endpoints.

\begin{figure}[htbp]
\begin{center}
 \scalebox{.50}{ \input{orthonet.pstex_t}}
 \caption{$L_{h}^{\ast}(v_i)$.}
\label{figure:orthonet}
\end{center}
\end{figure}

\medskip

Assume now that the disc $D(u,v)$ whose boundary is 
$L([u,v])\cup M([u,v])$ is to the {\it left} or to the {\it right} of $L(v_i)$.
Each level curve of $h^{\ast}$ is simple, joins  ${\mathcal Q}_{\mbox\small{\rm base}}$ to 
${\mathcal Q}_{\mbox\small{\rm base}}$, and the union of which foliates ${\mathcal Q}_{\mbox\small{\rm slit}}$.

Since ${\mathcal Q}_{\mbox\small{\rm slit}}$ is planar, standard arguments employing the Jordan Curve Theorem imply that there exists (at least) one level curve  of $h^{\ast}$, $L_{h}^{\ast}(v_i)$, close (in the Hausdorff distance) to $L(v_i)$, which intersects $M(v_j)$ in at least two points. This contradicts the assertion of Proposition~\ref{th:topooflevelofhasth}. 
The case in which $L([u,v])=M([u,v])$ follows by a simple modification of the above argument.

\medskip

Analogous cases, in which the disc $D(u,v)$  lies {\it under} or {\it above}  $L(v_i)$, are treated by employing a nearby (to $M(v_j)$) level curve of $g^{\ast}$ which yields a contradiction to  the assertion of Proposition~\ref{th:topooflevel}.  

\end{proof}

\medskip

\eop{\ref{th:netisrecyangular}}

Thus, the proof of Theorem~\ref{th:pair is orth} is now complete.

\eop{\ref{th:pair is orth}}


\subsection{The pair-flux length} 
\label{se:fgmetric}

In this subsection, we will define a notion of \emph{length} for the level curves of $g$ and of $g^{\ast}$.  To give some perspective, recall that 
in \cite{Du}, Duffin defined a metric to be a function $\tau:E\ra[0,\infty]$. More recently, 
in \cite{Ca}, Cannon defined a {\it discrete metric} to be a function $\rho : V\ra [0,\infty)$. The length of a path is then given by integrating $\tau,\rho$ along it, respectively.  

\smallskip

In \cite[Definition 1.9]{Her1}, we defined a metric (in Cannon's sense) which utilized  $g$,  the solution of the boundary value problem, alone.  
 
\smallskip
 
Motivated by the planar Riemannian case (see Equation (\ref{eq:polarmetric})), we will define new notions of metric and length for level curves of $g$ in ${\mathcal Q}_{\mbox\small{\rm slit}}$ and  thereafter in ${\mathcal A}$. These notions will incorporate both $g$ and $g^{\ast}$. (We will of course consider these notions for level curves that are given in their minimal form.) 
 
\smallskip 

\begin{Def}
\label{de:pairflux}
With the notation of the previous sections, we define the following:
\begin{enumerate}
\item For $e=[e^{-},e^{+}]$,  
let $\psi(e)=e^{-}$ be the map which associates to an edge its initial vertex. The {\it pair-flux weight}  of $e$  is  defined by
\begin{equation}
\label{eq:gradient metric}
\rho(e)=  \frac{2 \pi}{\mbox{\rm period}(g^{\ast}) } \exp\big( \frac{2 \pi}{\mbox{\rm period}(g^{\ast}) }    g(\psi(e))\big) |dh(e)|
=
\frac{2 \pi}{\mbox{\rm period}(g^{\ast}) } \exp\big( \frac{2 \pi}{\mbox{\rm period}(g^{\ast}) }   g(e^{-}))\big) |dh(e)|, 
\end{equation}
where $dh(e)= h(e^{+})-h(e^{-})$.
\smallskip
\item
Let $L$ be any path in ${\mathcal R}$; then its {\it length} with respect to the pair-flux weight is given by integrating $\rho$ along it, 
\begin{equation}
\label{eq:length gradient metric1}
\mbox{\rm Length}(L)=\int_{e\in L} \rho(e).
\end{equation}
\end{enumerate}
\end{Def}
\smallskip
In the applications of this paper, we will use the pair-flux weight  to provide a notion of length to level curves of $g$. Thus, by the assertion of Proposition~\ref{pr:topisalevelcurve}, we may now deduce 
\begin{Cor}
\label{co:lengthlevelg}
Let $L(v)$ be a closed level curve of $g$ (oriented counter-clockwise), and let $0\leq m=g({L(v)})\leq k$; then we have
\begin{equation}
\label{eq:levelg}
\mbox{\rm Length}(L(v))= 2 \pi \exp\big(\frac{2 \pi}{\mbox{\rm period}(g^{\ast})} \,m\big).
\end{equation}
\end{Cor}

\medskip

The definition of the pair-flux length is one of the new advances of this paper. Whereas in \cite{Her1,Her2,Her3} other notions of lengths utilizing only the solution $g$ were introduced, the  pair-flux length incorporates the pair $\{g,h\}$.  
This appealing feature is motivated by the case in 
the smooth category, i.e., for $z=r\exp{(i\theta)}$ in the complex plane, we have
\begin{equation}
\label{eq:polarmetric}
dz=i r\exp{(i\theta)}d\theta +\exp{(i\theta)} dr.
\end{equation}

\medskip

We will now provide a notion of length to the level curves of $h$ in ${\mathcal Q}_{\mbox\small{\rm slit}}$,  and  thereafter in ${\mathcal A}$. 
Keeping the analogy with the planar Riemannian case, the restriction of the Euclidean length element to level curves of the function $v(r,\phi)=\phi_0$, has the form
\begin{equation}
\label{eq:radial}
|d z|= dr.
\end{equation}

\begin{Def}
\label{de:lengthfortheata}
Let $L(h)=(v_0,\ldots,v_k)$ be a level curve of $h$ with $v_0\in E_2$ and $v_k\in E_1$, then its length  is given by 
\begin{equation}
\label{eq:length gradient metric2}
\mbox{\rm Length}(L(h))=\exp(g(v_k))-\exp(g(v_0))=\exp(k)-1. 
\end{equation}
\end{Def}

\medskip

\begin{Rem}
It is a consequence of Proposition~\ref{pr:rightisalevelcurve} that any two level curves of $h$ have the same length.
\end{Rem}

\bigskip


\section{The cases of an annulus and an annulus with one singular boundary component}
\label{se:AnnandQaud}

\subsection{The case of an annulus}
\label{se:Annulus}
In this subsection, we study the important case of an annulus. It is the first case, in terms of the connectivity of the domain $\Omega$, of  the one described in  Definition~\ref{de:dbvp} (Subsection~\ref{pa:bvp}). Let  ${\mathcal R}$ be the rectangular net associated with the combinatorial orthogonal filling pair 
$\{g, h\}$ which was constructed in the proof of Theorem~\ref{th:pair is orth}.

\smallskip

We use the term  {\it measure} on the space of quadrilaterals in ${\mathcal R}^{(2)}$ to denote a non-negative set function defined on ${\mathcal R}^{(2)}$. An example of such, which will be used  in Theorem~\ref{th:annulus}, is provided  in 

\begin{Def}
\label{de:combdirichletmeasure}
For any $R\in {\mathcal R}^{(2)}$, let $R_{\mbox\small{\rm top}}, R_{\mbox\small{\rm base}}$ be the pair of its opposite boundaries that are contained in successive level sets of $g$; we will denote them by the 
top and base boundaries of $R$, respectively (where the top boundary corresponds to a larger value of $g$).  Let 
$t\in R_{\mbox\small{\rm top}}^{(0)}$ and $b\in  R_{\mbox\small{\rm base}}^{(0)}$ be any two vertices. Then we let 

\begin{equation}
\label{eq:dfm}
\nu(R)=\frac{1}{2}\big(\exp^{2}(\frac{2 \pi}{\mbox{\rm period}(g^{\ast})}g(t))  -  \exp^{2}(\frac{2 \pi}{\mbox{\rm period}(g^{\ast})}g(b))    \big) \frac{2 \pi\, dh(R_{ \mbox\small{\rm base} })}{\mbox{\rm period}(g^{\ast})}. 
\end{equation}
 \end{Def}

\begin{Rem}
By the construction of ${\mathcal R}$, all the vertices in $R_{\mbox\small{\rm top}}\  (R_{\mbox\small{\rm base}})$ have the same $g$ values and $dh(R_{ \mbox\small{\rm base} })= dh(R_{ \mbox\small{\rm top} })$. 
\end{Rem}

\smallskip

We now turn to the

\medskip

\noindent{\bf Proof of Theorem~\ref{th:annulus}}.
Recall (see the discussion preceding the proof of Theorem~\ref{th:netisrecyangular}) 
 that the vertices in ${\mathcal R}^{(0)}$ are comprised of all the intersections of the level curves of $g$ (the family ${\mathcal L}$) and the level curves of $h$ (the family ${\mathcal M}$). Thus, the  vertex $(i,j)$ will denote the unique vertex determined by the intersection of $L(v_i)$ and 
$M(v_j)$; the existence and uniqueness of this intersection are consequences of
Theorem~\ref{th:netisrecyangular}.

\medskip

The harmonic conjugate function $h$ is single-valued on ${\mathcal Q}_{\mbox\small{\rm slit}}$, and multi-valued with a period which is equal to $\mbox{\rm period}(g^{\ast})$, when extended to  ${\mathcal A}$. This means that 
\begin{equation}
\label{eq:periodinannulus}
h(z_1)=h(z_0)+\mbox{\rm period}(g^{\ast}), 
\end{equation}
whenever $z_1\in L(z_0)$ is obtained from $z_0\in\mbox\small{\rm slit}({\mathcal A})$ by traveling one full cycle along the $g$- level curve $L(z_0)$. Hence, the  function
\begin{equation}
\label{eq:welldefined}
\frac{2\pi}{\mbox{\rm period}(g^{\ast})}\big( g(v)   + i h(v) ),\ v\in {\mathcal A}\cap{\mathcal R}^{(0)}
\end{equation}
has period $2\pi i$ when defined on ${\mathcal A}$. Therefore,
\begin{equation}
\label{eq:welldefined1}
\exp\big(\frac{2\pi}{\mbox{\rm period}(g^{\ast})}\big( g(v)   + i h(v) )\big),\ v\in {\mathcal A}\cap{\mathcal R}^{(0)}
\end{equation}
is single-valued on  ${\mathcal A}$.

\medskip
We now turn to the construction of the tiling $T$.
\medskip

 The tiling $T$  of $S_{\mathcal A}$ is determined by all the intersections
of the family of concentric circles, ${\mathcal C}$, defined by 
\begin{equation}
\label{eq:circcurves}
r_i =\exp\big(\frac{2\pi}{\mbox{\rm period}(g^{\ast})} \,  g(v_i)\big),\ \mbox{\rm for}\   i=0,\ldots,k, 
\end{equation}
with the family of radial lines $\Gamma$, defined by
\begin{equation}
\label{eq:radialcurves}
\phi_j= \frac{2\pi}{\mbox{\rm period}(g^{\ast})}\, h(v_j),\ \mbox{\rm for}\  j=0,\ldots,p,
\end{equation}
where each annular shell in the tiling is uniquely defined by four vertices that lie on two consecutive members of the families above.
 
\medskip

Let $f_{\mathcal R}$ be a homeomorphism  which maps the quadrilateral $R\in{\mathcal R}$ determined by the counterclockwise oriented ordered set of vertices 
\begin{equation}
\label{eq:vertices}
\{(i,j), (i+1,j),(i+1,j+1),(i,j+1)\} 
\end{equation}
for $i=0,\ldots,k-1$, and $j=0,\ldots,p-1$, 
onto the annular shell $T_{R}$ determined by the counterclockwise oriented ordered set of vertices
\begin{equation}
\label{eq:vetricesshell}
\{r_i \exp{(i\phi_j)} , r_{i+1}\exp{(i\phi_{j})}, r_{i+1}\exp{(i\phi_{j+1})} ,  r_i\exp{(i\phi_{j+1})}\}
\end{equation}
and that preserves the order of the vertices.

\begin{figure}[htbp]
\begin{center}
 \scalebox{.55}{ \input{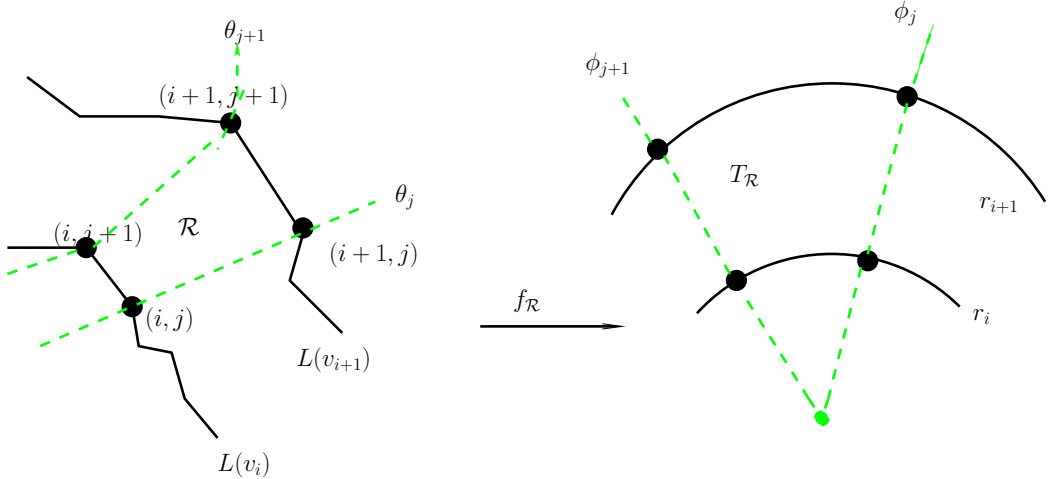}}
 \caption{Constructing one annular shell.}
\label{figure:quad4}
\end{center}
\end{figure}

\medskip
\medskip

We will now show how to choose the  $f_R$'s so that the induced extension of $f$ is a homeomorphism on the whole of $\Omega$. Let $Q$ and $R$ be adjacent quadrilaterals in the layer between $E_2$ and $L(v_1)$. One can easily show that  by choosing $f_R$ to agree  with  $f_Q$ on $Q\cap R$, $f_Q\cup f_R$  is a homeomorphism on their intersection. Continuing counterclockwise in the fashion described above, it follows that 
 the full layer between  $E_2$ and $L(v_1)$ is mapped via a homeomorphism onto its image. 

Let  $\bar Q$ be the adjacent quadrilateral to $Q$ from above. Choose $f_{\bar Q}$  so that it agrees with $f_Q$ on $\bar Q \cap Q$.  Continue counterclockwise so that for each quadrilateral the 
chosen homeomorphism agrees with the choice of the previous homeomorphism (in this layer) on its right edge, and also agrees on its base, with the choice of the homeomorphism for the quadrilateral under it (on its top). 

For the last quadrilateral in this layer, a homeomorphism can be chosen to agree with the first chosen homeomorphism (in this layer) on its left side, with the homeomorphism chosen before it on its right side, and agrees (on its base) with the homeomorphism chosen for the quadrilateral lying under it. We continue this process, a layer by layer, until the domain is exhausted. 

Let 
\begin{equation}
\label{de:definitionoff}
f=\bigcup_{R\in {\mathcal R}^{(2)}} f_R.
\end{equation}

It is clear from (\ref{eq:vertices}) and (\ref{eq:vetricesshell}) that $f$ maps any edge in ${\mathcal R}^{(1)}\cap \partial\Omega$ homeomorphically onto an arc in $\partial S_{\mathcal A}$. Therefore, $f$ is indeed boundary preserving.

\medskip

It remains to prove that the map $f=\cup_R f_R$ which is  assembled from the individual maps as defined above is a homeomorphism {\it onto} $S_{\mathcal A}$.  To this end, first observe that by the maximum principle, the map $f$ is into $S_{\mathcal A}$. 

\medskip

By the definitions of $\nu,\mu$ and $T_R$, we have for all $R\in {\mathcal R}^{(2)}$ that
\begin{equation}
\label{eq:measure}
\nu(R)=\mu(T_{R}). 
\end{equation}

It is clear from the construction of ${\mathcal R}$ that any two quadrilaterals in ${\mathcal R}^{(2)}$ have disjoint interiors and that their intersection is either a single vertex or a common edge. Also recall that by definition each quadrilateral has its top and bottom edges situated on two successive level curves in ${\mathcal L}$, and its right and left edges situated on two successive level curves in ${\mathcal M}$.

\smallskip

Since the union of the quadrilaterals in ${\mathcal R}^{(2)}$  tile ${\mathcal A}$, the  total $\nu$-measure of their union, which we define to be $\nu({\mathcal A})$, satisfies the following 
 
\begin{equation}
\label{eq:totalmeasure}
\nu({\mathcal A}) \equiv  \nu(\cup_{R\in{\mathcal R}^{(2)}}R)=\sum_{R\in {\mathcal R}^{(2)} }\nu(R).
\end{equation}

\medskip

Starting from the quadrilaterals that lie between $L(v_0)$ and $L(v_1)$, we sum the $\nu$-measure of all the quadrilaterals in the layer defined in between successive level curves of $g$, until we reach $L(v_k)$.  By employing Definition~\ref{de:combdirichletmeasure} an easy computation shows that

\begin{equation}
\label{eq:totalcombmeasure}
\nu({\mathcal A})= \pi  \big(   \exp^{2}\big( \frac{2\pi}{ \mbox{\rm period}(g^{\ast})}\,k  \big) - 1           \big).
\end{equation}

\medskip

By the construction of the annular shells and the definition of the map $f$, each quadrilateral $R$ is mapped onto a unique annular shell $T_R$. No two different quadrilaterals are mapped onto the same annular shell, and the collection of their images tiles a subset of  $S_{\mathcal A}$. 

\medskip

Hence, by applying the above paragraph, (\ref{eq:totalmeasure}), (\ref{eq:measure}), (\ref{eq:totalcombmeasure}) and the definition of $S_{\mathcal A}$, we obtain that
\begin{equation}
\label{eq:areacompare}
\nu({\mathcal A})= \sum_{R\in {\mathcal R}^{(2)} }\mu(T_R)=\mu(\bigcup_{R\in {\mathcal R}^{(2)}}T_R)=    \mu(S_{{\mathcal A}}).
\end{equation}
\medskip

Hence, there are no gaps nor overlaps in the tiling of $S_{\mathcal A}$ and therefore $f$ is onto.

\smallskip
This concludes the proof of the theorem.

\eop{}

\medskip

\begin{Rem} The proof shows that each curve in  set ${\mathcal L}$ is  mapped homemorphically onto a  (unique) level curve in the family $u(r,\phi)=r_i$, and that each curve in the set
${\mathcal M}$ is mapped homeomorphically onto a (unique) level curve in the family $v(r,\phi)=\phi_j$. Also, the discussion following Definiton~\ref{de:period} guarantees that the dimensions of $S_{\mathcal A}$ are independent of the choice of the slit chosen.
\end{Rem}

\medskip

\subsubsection{\bf Relation of Theorem~\ref{th:annulus} to works by Schramm  and Cannon-Floyd-Parry.}
\label{se:CannonSchramm}

It is imperative to relate this theorem to Theorem 1.3 in \cite{Sch}, and Theorem 3.0.1 in \cite{CaFlPa}. While Schramm, and Cannon, Floyd, and Parry used discrete extremal lengths arguments in their proofs, their  arguments as well as  their results are different. Schramm's proof seems to work for a quadrilateral but not directly for an annulus.  The methods of Cannon, Floyd and Parry work for both a quadrilateral and an annulus. Furthermore, Schramm's input is a triangulation with a contact graph that will (more or less) be preserved. The input for Cannon, Floyd and Parry is more flexible. They consider a covering of a topological quadrilateral (annulus) by topological disks. We refer the reader to the papers above for details. Upon applying a Dirichlet-Neumann boundary value problem, our methods of the proof of Theorem~\ref{th:annulus} may be adapted to work for the quadrilateral case as well.

\smallskip

While our proof of Theorem~\ref{th:annulus} does not use the machinery of extremal length arguments, it is worth recalling that in the smooth category there are celebrated connections between boundary value problems and extremal length (see for instance \cite[Theorem 4.5]{Ah}).

\smallskip
The common theme of our methods and those of Cannon, Floyd and Parry in \cite{CaFlPa} is the construction of a new coordinate system on a topological annulus. As stated in the introduction, this powerful idea goes back to Riemann.

\subsection{The case of an annulus with one singular boundary component}
\label{se:singbound}
In this subsection, we will generalize Theorem~\ref{th:annulus} by providing a geometric model for an annulus with one singular boundary component. The singular  boundary component is of a special type.  It is determined by the topological structure of a singular level curve of the solution of a Dirichlet boundary value problem imposed on a planar embedded $m$-connected domain, where $m>1$. 

\medskip
We start with two definitions; the first one appeared in \cite[page 9]{Her1}.

\medskip

\begin{Def}
\label{de:bouquet}
A generalized  bouquet of circles will denote a union of bouquets of  piecewise-linear circles  where the intersection of any two circles is at most a vertex. 
Moreover, all such tangencies are required to be exterior, i.e., no circle is contained in the interior
of the bounded component of another. 
\end{Def}

Recall that Theorem 2.15, which was proved in \cite{Her1}, asserts the following.
\begin{Thm}[The topology of a level curve]
\label{th:notsimple}
Let $L$ be a level curve for $g$. Then each connected component of $L$ is a generalized bouquet of  circles. 
\end{Thm}

\medskip

It is convenient to present the singular boundary component as a quotient space. In the following definition, a circle will mean either a round circle or a piecewise linear circle.

\smallskip

We are now ready to make 

\begin{Def}
\label{de:circwithlabel}
An embedded planar circle with finitely many distinguished points on it will be called a {\it labeled} circle. If in addition, equivalence relations among these points are given, so that the quotient of the labeled circle is a generalized bouquet of circles, then we call the quotient a {\it labeled bouquet}, and the labeled circle will be called {\it good}. 
\end{Def}

\begin{Rem}
\label{re:smoothbouquet}
Note that if a labeled round bouquet, i.e., one which consists of only round circles, has more than two round circles tangent at one point, it will no longer embed in $\mathbb{R}^{2}$. 
\end{Rem}

We will now define the object of study in this subsection. By a {\it generalized singular annulus}, ${\mathcal A}_{\mbox{\small\rm sing}}$, we will mean a subset of the plane, whose interior is homeomorphic to the interior of an annulus, and whose boundary has two components: one of which is homeomorphic to $\mathbb S^{1}$ and the other is a generalized bouquet of circles. The subscript denotes the set of tangency points in the generalized bouquet of circles. Let us also assume that a cellular decomposition ${\mathcal T}$ of $({\mathcal A}_{\mbox{\small\rm sing}},\partial{\mathcal A}_{\mbox{\small\rm sing}})$ is provided, where each $2$-cell is either a triangle or a quadrilateral.

\medskip

Topologically, ${\mathcal A}_{\mbox{\small\rm sing}}$ may be presented as the quotient 
of a planar annulus ${\mathcal A}$, where $\partial {\mathcal  A}= E_1\cup E_2$, and $E_2$ is a good labeled circle (see Definition~\ref{de:circwithlabel}).  Henceforth, we will let $\pi$ denote the quotient map. We will let $\dot{E_2}$ denote the 
singular boundary component of $\partial{\mathcal A}_{\mbox{\small\rm sing}}$. 

\medskip

Note that the cellular decomposition ${\mathcal T}$ can be  lifted to a cellular decomposition $\tilde{\mathcal T}$ of $({\mathcal A},\partial{\mathcal A})$, where each $1$-cell,  $2$-cell in $\tilde{\mathcal T}$, respectively, is the unique pre-image, under $\pi^{-1}$, of a unique $1$-cell, $2$-cell  in ${\mathcal T}$, respectively.  The difference between the two cellular decompositions manifests in the addition (in comparision to $\dot{E_2}$) of vertices in $E_2$. Specifically, for each vertex $v$ in the singular part of $\partial{\mathcal A}_{\mbox{\small\rm sing}}$, there are $m(v)$ vertices in $E_2$, where $m(v)$ is the number of circles that are tangent at $v$.

\medskip

We will now apply Theorem~\ref{th:pair is orth} and Theorem~\ref{th:annulus} to $({\mathcal A},\partial {\mathcal A}, \tilde{\mathcal T})$.  In the following proposition, recall that the existence of ${\mathcal R}$ is provided  by Theorem~\ref{th:pair is orth}, and that $h$ is the conjugate harmonic function to $g$, the solution of the imposed discrete Dirichlet boundary value problem on $({\mathcal A},\partial {\mathcal A}, \tilde{\mathcal T})$ (see Definition~\ref{de:dbvp}).

\smallskip

With the above notation and setting in place, and with ${\mathcal L}$ denoting the set of level curves of $g$ as in Equation~(\ref{eq:levelsets}), we may now state the main proposition of this subsection.

\begin{Pro}[An annulus with a singular boundary]
\label{th:annulussing} Let $({\mathcal A}_{\mbox{\small\rm sing}}, \partial{\mathcal A}_{\mbox{\small\rm sing}}=E_1\cup\dot{E_2})$ be a generalized singular annulus endowed with a cellular decomposition ${\mathcal T}$. Let $k$ be a positive constant, and let $g$ be the solution of the discrete Dirichlet boundary value problem defined on $({\mathcal A}, \partial {\mathcal A}, \tilde{\mathcal T})$. 

Let $S_{\mathcal A}$ be the concentric Euclidean annulus with its inner and outer radii satisfying 
\begin{equation}
\label{eq:dimannsing}
\{r_1,r_2\}=  \{1, 2\pi\,\mbox{\rm Length}(L(v_k))\}=\{1, 2\pi \exp\big( \frac{2\pi}{\mbox{\rm period}(g^{\ast})}\,     k\big)\}.
\end{equation}

Then there exist 

\begin{enumerate} 

\item a tiling $T$ of $S_{\mathcal A}$ by annular shells,
\item a set denoted by $\mbox{\rm sing}(S_{\mathcal A})$ consisting of finitely many points which is contained in the inner boundary of $S_{\mathcal A}$,
\item a homeomorphism $$f:({\mathcal A},\partial{\mathcal A} \setminus \pi^{-1}(\mbox{\rm sing(}{\mathcal A})) ,{\mathcal R})\rightarrow (S_{\mathcal A},\partial S_{\mathcal A}\setminus \mbox{\rm sing}(S_{\mathcal A}),T)
$$
such that 
 $f$ maps the interior of each quadrilateral in ${\mathcal R}^{(2)}$ onto the interior of a single annular shell in $S_{\mathcal A}$, $f$ preserves the measure of each quadrilateral, i.e., $$\nu(R)=\mu(f(R)),\mbox{\rm for all}\  R\in {\mathcal R}^{(2)},$$ and $f$ is boundary preserving.
\end{enumerate}
\end{Pro}

\begin{proof}
The proof is a straightforward modification of the non-singular boundary case. Let ${\mathcal R}$ be the rectangular net constructed in Theorem~\ref{th:pair is orth}.  Let $h$ be the conjugate harmonic function constructed on $({\mathcal A},\partial{\mathcal A},\tilde{\mathcal T})$,  let $f$ be the homeomorphism constructed in Theorem~\ref{th:annulus}, and  let $T$ be the tiling of $S_{\mathcal A}$ provided by Theorem~\ref{th:annulus}.

\smallskip

For each $t_i\in \mbox{\rm sing}({\mathcal A})\subset \dot{E_2}$, $i=1,\ldots p$, there are precisely $m(t_i)$ vertices on $E_2$ in the equivalence class corresponding to $t_i$. Let 
\begin{equation}
\label{eq:singularleveloftheta}
{\mathcal V}(t_i)=\{L(h)_{t_{i,1}},\ldots,L(h)_{t_{i,m(t_i)}}\},\  i=1,\ldots p
\end{equation}
be the level curves of $h$ that have one of their endpoints at one of these vertices. With this notation, and since the level curves of $h$ are ``parallel", it follows that
\begin{equation}
\label{eq:totalsing}
{\mathcal V}_{\mbox{\rm sing}({\mathcal A})}= \bigcup_{i=1}^{p} {\mathcal V}(t_i)
\end{equation}
comprises of all the level curves of $h$ that have an endpoint in the pre-image of $\mbox{\rm sing}({\mathcal A})$.

\smallskip

Set 
\begin{equation}
\label{eq:singincylinder}
\mbox{\rm sing}(S_{\mathcal A})=f(E_2)\bigcap f(\bigcup_{i=1}^{p} {\mathcal V}(t_i) ),
\end{equation}
then $\mbox{\rm sing}(S_{\mathcal A})$ is the image under $f$  of all the vertices in the pre-image of 
$\mbox{\rm sing}({\mathcal A})$. Furthermore, recall that $f( {\mathcal V}_{\mbox{\rm sing}(S_{\mathcal A})})$ is a set of radial arcs in $S_{\mathcal A}$.

\smallskip
To finish proving the statement in $(3)$, note that any quadrilateral in ${\mathcal R}^{(2)}$ whose vertices are disjoint from  $\mbox{\rm sing}(S_{\mathcal A})$ is mapped homemorphically onto a shell in $S_{\mathcal A}$. Since by construction the image of  $\pi^{-1}(\mbox{\rm sing(}{\mathcal A}))$ is precisely $\mbox{\rm sing}(S_{\mathcal A})$,  it follows that $f$ will map the interior of each one of the rest of the quadrilaterals homemorphically onto the appropriate shell, with punctures at the corresponding vertices. This ends the proof of the proposition. 
\end{proof}

\smallskip

A geometric model to $({\mathcal A}_{\mbox{\rm sing}},\partial{\mathcal A}_{\mbox{\rm sing}}, {\mathcal T})$ is now easy to provide since 
 the first part of $(3)$ in the  proposition above allows us to label the vertices in $\mbox{\rm sing}(S_{\mathcal A})$ isomorphically to the labeling of the vertices in $\mbox{\rm sing}({\mathcal A})$.  We will keep denoting by $\pi$ the quotient map which is thereafter induced on $S_{\mathcal A}$. Such a quotient annulus will be called a {\it generalized Euclidean annulus} and will be denoted by $C_{\mathcal A}$.
The proof of the following corollary is straightforward.

\begin{Cor}
\label{co:to the quotient}
With the assumptions of Proposition~\ref{th:annulussing}, and with $C_{\mathcal A}=S_{\mathcal A}/\pi$, there exist

\begin{enumerate} 

\item a tiling $T$ of $C_{\mathcal A}$ by annular shells, and
\item a homeomorphism $$f:({\mathcal A}_{\mbox{\rm sing}},\partial{\mathcal A}_{\mbox{\rm sing}} ,{\mathcal R})\rightarrow (C_{\mathcal A},\partial C_{\mathcal A},T),
$$
such that 
$f(\mbox{\rm sing}({\mathcal A}))=\mbox{\rm sing}(S_{\mathcal A})/\pi$,  $f$ maps the interior of each quadrilateral in ${\mathcal R}^{(2)}$ onto the interior of a single annular shell in $C_{\mathcal A}$, $f$ preserves the measure of each quadrilateral, i.e., $$\nu(R)=\mu(f(R)),\mbox{\rm for all}\  R\in {\mathcal R}^{(2)},$$ and $f$ is boundary preserving.
\end{enumerate}
\end{Cor}

\smallskip

\begin{figure}[htbp]
\begin{center}
 \scalebox{.55}{ \input{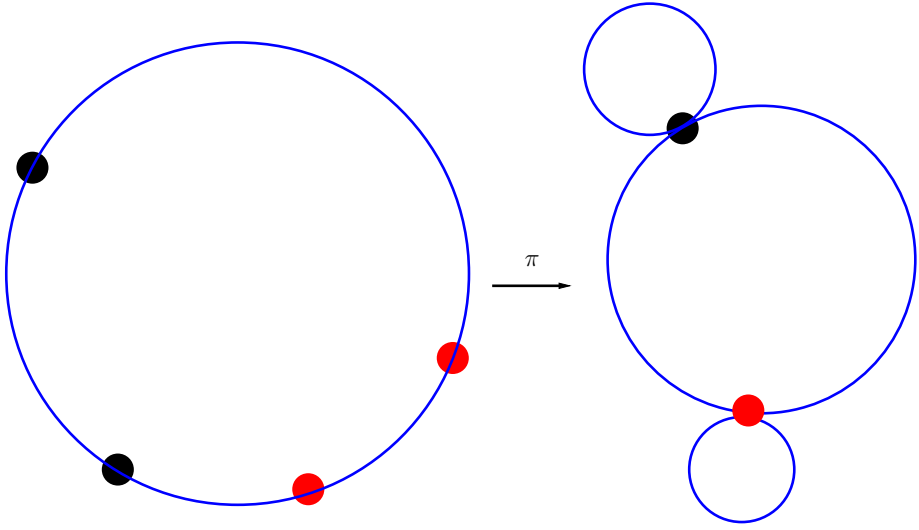}}
 \caption{An example of the map $\pi$.}
\label{figure:quad2}
\end{center}
\end{figure}

\bigskip

In the next section, we will work with a general $m$-connected planar domain ($m>2$) that will
be cut along singular level curves of a Dirichlet boundary value problem imposed on it.  In order to allow gluing along components of a singular level curve, we will utilize Euclidean cylinders and Euclidean cylinders with one singular boundary component. To this end, recall that a conformal homeomorphism, from a concentric annulus to a Euclidean cylinder of radius equal to $1$, and height equal to $\log{(b/a)}$, is defined by 

\begin{equation}
\label{eq:annulustocylinder}
F(\rho\cos{(h)},\rho\sin{(h)})=(\cos{(h)},\sin{(h)}, \log{(\rho)}),\ a\leq\rho\leq b, \ 0<h\leq 2 \pi, 
\end{equation}
where $(\rho,h)$ denote polar coordinates in the plane.

It easily follows that the image of an annular shell under the map $F$ is a Euclidean rectangle. We will abuse notation and will keep the same notation for ${\mathcal S}_{\mathcal A}$ and its image under the mapping $F$. 
\smallskip

We now define a variation of the measure $\nu$ (see Equation~(\ref{eq:dfm})) in order to adjust our statements to working with such cylinders.

\begin{Def} 
\label{de:measoncyl} For any $R\in {\mathcal R}^{(2)}$,  let $R_{\mbox\small{\rm top}}, R_{\mbox\small{\rm base}}$ be the top and base boundaries of $R$, respectively.  Let 
$t\in R_{\mbox\small{\rm top}}^{(0)}$ and $b\in  R_{\mbox\small{\rm base}}^{(0)}$ be any two vertices. Then let  
\begin{equation}
\label{eq:measoncylinder}
\lambda(R)= \frac{2 \pi\, dh(R_{ \mbox\small{\rm base} })}{\mbox{\rm period}(g^{\ast})}
\log{\frac{r_{t}}{r_{b}}},
\end{equation}
where 
\begin{equation}
\label{eq:circcurves1}
r_t =\exp\big(\frac{2\pi}{\mbox{\rm period}(g^{\ast})} \,  g(t)\big)\ \mbox{\rm and}\   r_b =\exp\big(\frac{2\pi}{\mbox{\rm period}(g^{\ast})} \,  g(b)\big), 
\end{equation}
following Equation~(\ref{eq:circcurves})
\end{Def}

\medskip
By applying the map $F$ and the measure $\lambda$, we 
 may state  
Theorem~\ref{th:annulus}, Proposition~\ref{th:annulussing} and Corollary~\ref{co:to the quotient} in the language of Euclidean cylinders. We end this subsection by summarizing this in the following remark which will be applied in the next section.

\medskip

\begin{Rem}
\label{re:to a cylinder}
Under the assumptions of Theorem~\ref{th:annulus}, Proposition~\ref{th:annulussing}, and Corollary~\ref{co:to the quotient}, all the assertions hold if  one replaces $S_{\mathcal A}$,   generalized  $S_{\mathcal A}$,  by a Euclidean cylinder, generalized Euclidean cylinder, respectively; $f$ by $\pi\circ F\circ f$ and an annular shell by its image under $F$, $F\circ \pi$, respectively, and the measure $\nu$ by the measure $\lambda$. 
\end{Rem}

\section{planar domains of higher connectivity} 
\label{se:highconn} 

In this section, we prove the second main theorem of this paper. We generalize Theorem~\ref{th:annulus} to the case of bounded planar domains of higher connectivity.
Let us start by recalling an important property of the level curves of the solution of the discrete Dirichlet boundary value problem (see Definition~\ref{de:dbvp}). This property will be essential in the proof of Theorem~\ref{th:DBVP}.  In the course of the proof, we will need to know that there is a singular level curve which encloses all of the interior components of $\partial\Omega$, where $\Omega$ is the given domain. This unique level curve is the one along which we will cut the domain. We will keep splitting along a  sequence of these singular level curves in subdomains of smaller connectivity until the remaining pieces are annuli or generalized singular annuli.  Once this is achieved, we will provide a gluing scheme in order to fit the pieces together in a geometric way. 

\medskip

Before stating the second main theorem of this paper, we need to recall a definition and a proposition. Consider $f:V \rightarrow \mathbb{R}\cup \{0\}$ such that any two adjacent vertices are given different values. Let  $\{w_1,w_2,\ldots,w_k\}$ be the adjacent vertices to $v\in V$.
Following \cite{Ba} and \cite[Section 3]{LaVe}, consider the number of sign changes in the sequence $\{ f(w_1)-f(v),f(w_2)-f(v),\ldots,f(w_k)-f(v),f(w_1)-f(v)\}$, which is denoted by  $\mbox{\rm Sgc}_{f}(v)$. The index of $v$ is then defined by 
\begin{equation}
\label{eq:index}
\mbox{\rm Ind}_{f}(v)= 1- \frac{\mbox{\rm Sgc}_{f}(v)}{2}.
\end{equation}

\begin{Def}
A vertex whose index is different from zero will be called  singular; otherwise the vertex is regular. A level set which contains at least one singular vertex will be called singular; otherwise the level set will be called regular.
\end{Def} 

The following proposition first appeared (as Proposition 2.28) in \cite{Her1}.

\begin{Pro}
\label{pr:onethatenclose}
There exists a unique singular level curve which contains, in the  interior of the domain it bounds, all of the inner boundary components of $\partial\Omega$. 
\end{Pro}
Such a curve will be called the {\it maximal} singular level curve with respect to $\Omega$.
Recall that the notion of an interior of such a domain was discussed in Subsection~\ref{pa:simple}. 

\smallskip

Throughout this paper, we will not distinguish between a Euclidean rectangle and its image under an isometry. Recall (see the end of Subsection~\ref{se:gmodels}) that a singular flat, genus zero compact surface with $m>2$ boundary components with conical singularities is called  a ladder of singular pairs of pants.

\smallskip
We now  prove the second main theorem of this paper.  
\begin{Thm}[A Dirichlet model for an $m$-connected domain]
\label{th:DBVP}
Let  $(\Omega,\bord\Omega=E_1\sqcup E_2,{\mathcal T})$ be a bounded, $m$-connected, planar domain with $E_2=E_2^1\sqcup E_2^2\ldots \sqcup E_2^{m-1}$. Let $g$ be the solution of the discrete Dirichlet boundary value problem defined on $(\Omega, \partial {\Omega}, {\mathcal T})$. Then there exists

\begin{enumerate}
   \item a finite decomposition with disjoint interiors of $\Omega$,  ${\mathcal A}=\cup{_i}{\mathcal A}_{i}$, where for all $i$, ${\mathcal A}_{i}$ is either an annulus or an annulus with one singular boundary component;
  
  \smallskip
   
\item for all $i$, a finite decomposition with disjoint interiors ${\mathcal R_{A_i}}$, of ${\mathcal A}_{i}$,  where each $2$-cell is a simple quadrilateral;

\smallskip

\item for all $i$, a finite measure $\lambda_i$ defined on ${\mathcal R_{A_i}}$; and 

\smallskip

\item a ladder of singular pairs of pants  $S_{\Omega}$ with $m$ boundary components, such that

\medskip

\begin{itemize}
\item[(a)] the lengths of the $m$ boundary components of $S_\Omega$ are determined by the Dirichlet data,
\item[\rm(b)] there exists a finite decomposition with disjoint interiors of $S_{\Omega}=\cup_{i} C_{{\mathcal A}_{i}}$, where 
each $C_{{\mathcal A}_{i}}$ is either a Euclidean cylinder or a  generalized Euclidean cylinder, equipped with 
 a tiling $T_i$ by Euclidean rectangles where each one of these is  endowed with Lebesgue measure;  and 
\item[\rm(c)] a homeomorphism
$$f:(\Omega,\bord\Omega,\cup{_i}{\mathcal A}_{i})\rightarrow (S_{\Omega},\partial \Omega,\cup_{i}{{\mathcal R_A}}_{i}),
$$
such that 
$f$ maps each ${\mathcal A}_{i}$ homeomorphically onto a corresponding $C_{{\mathcal A}_{i}}$,  and each quadrilateral in ${\mathcal R_A}_{i}$ onto a rectangle in 
$C_{{\mathcal A}_{i}}$ while preserving its measure. Furthermore, $f$ is boundary preserving (as explained in Theorem~\ref{th:annulus}).
\end{itemize} 

\end{enumerate}
\end{Thm}
\begin{proof}
The first part of the proof is based on a splitting scheme along a family of singular level curves of $g$ which will be proven to terminate after finitely many steps.  We will describe in detail the first two steps of the scheme, explain why it terminates, and leave the ``indices" bookkeeping required in the formal inductive step to the reader.  The outcome of the first part of the proof is a scheme describing a splitting of the top domain, $\Omega$, to simpler components, annuli and singular annuli. 

\smallskip

The complement of $L(\Omega)$, the maximal singular curve in $\Omega$, has at most $m$-connected components, all of which, due to Proposition~\ref{pr:onethatenclose}, have connectivity which is at most  $m-1$, or are annuli, or generalized singular annuli.  By the maximum principle, one of these components has all of its vertices with $g$-values that are greater than the $g$-value along $L(\Omega)$. In Subsection~\ref{pa:simple}, such a domain was denoted by
${\mathcal O}_{2}(L(\Omega))$ and was called an exterior domain. Its boundary consists of  $E_1$ and $L(\Omega)$. It follows from Proposition~\ref{pr:onethatenclose} and  Theorem~\ref{th:notsimple}  that it is a generalized singular annulus which will be denoted by ${\mathcal A}(E_1,L(\Omega))$.  

\smallskip

Let the full list of  components   of the complement of  $L(\Omega)$ in $\Omega$   be enumerated as 
\begin{equation}
\label{eq:list}
CC_1=\{CC_{1,1}(L(\Omega)),CC_{1,2}(L(\Omega)),\ldots, CC_{1,p}(L(\Omega))={\mathcal A}(E_1,L(\Omega)) \}.
\end{equation} 

\smallskip

By definition, for each $j=1, \ldots ,p-1$, the $g$-value on the boundary component 
\begin{equation}
\label{eq:procedure}
\partial_{1,j}= \partial CC_{1,j}(L(\Omega))\cap L(\Omega)
\end{equation}
is the constant which equals the $g$-value on $L(\Omega)$. The other components of $\partial CC_{1,j}(L(\Omega))$, $j=1, \ldots ,p-1$, are kept at $g$-values equal to $0$. Hence, we now impose a (discrete) Dirichlet boundary value problem with these values on each element in the list $CC_1\setminus {\mathcal A}(E_1,L(\Omega))$.
On ${\mathcal A}(E_1,L(\Omega))$, the induced Dirichlet boundary value problem is determined by the value of $g$ restricted to $E_1$ (which is equal to $k$), and the value of $g$ restricted to $L(\Omega)$. Note that imposing these boundary value problems in general will require introducing vertices of type I and of type II and changing conductance constants along new edges, as described in Subsection~\ref{pa:simple}. These modifications are done in such a way that the restriction of the original $g$ solves the new boundary value problems.

\smallskip

For $j=1,\ldots, p-1$, let $k_j$ denote the connectivity of $CC_{1,j}(L(\Omega))$.
We now repeat the procedure described in the first paragraph of the proof in each one of the connected components $CC_{1,j}(L(\Omega)),$ at most $k_j - 2$ times,  for $j=1,\ldots, p-1$, excluding those indices that correspond to annuli. 

\medskip

We will now describe the second step of the splitting scheme. For each $j\in\{1,\ldots, p-1\}$ whose corresponding component is not an annulus,  a maximal singular level curve 
\begin{equation}
\label{eq:singcurve}
L_j(CC_{1,j})=L(CC_{1,j}(L(\Omega)))
\end{equation}
with respect to the component  $CC_{1,j}(L(\Omega))$,   
is chosen. This is possible because at the end of the previous step, we imposed a Dirichlet boundary value problem on each one of these domains. Hence, the assertions of Proposition~\ref{pr:onethatenclose} and Theorem~\ref{th:notsimple} may be applied to these domains as well.

\smallskip

Therefore, a new list consisting of
connected components of the complement of $L_{j}(CC_{1,j})$ in $CC_{1,j}(L(\Omega))$,
of cardinality at most $m-1$,
 \begin{equation}
\label{eq:newlist}
CC_{1,j}=\{CC_{1,j,1}(L_j(CC_{1,j})),CC_{1,j,2}(L_j(CC_{1,j})),\ldots, CC_{1,j,v}(L_j(CC_{1,j}))\}, 
\end{equation}
$j$ as chosen above is generated.
We will let the last element in this list denote the exterior domain  to $L_j(CC_{1,j})$
in $CC_{1,j}$. It is, as in the first step of the scheme, a generalized singular annulus denoted by 
${\mathcal A}(\partial_{1,j},
L_j(CC_{1,j}))$. The other components have connectivity which is at most  $(m-2)$, or are annuli. Note that in this step the exterior domain from the first step in the scheme, $CC_{1,p}(L(\Omega))={\mathcal A}(E_1,L(\Omega))$, is left without any further splitting, since it is a generalized singular annulus. 

\smallskip

\begin{figure}[htbp]
\begin{center}
 \scalebox{.55}{ \input{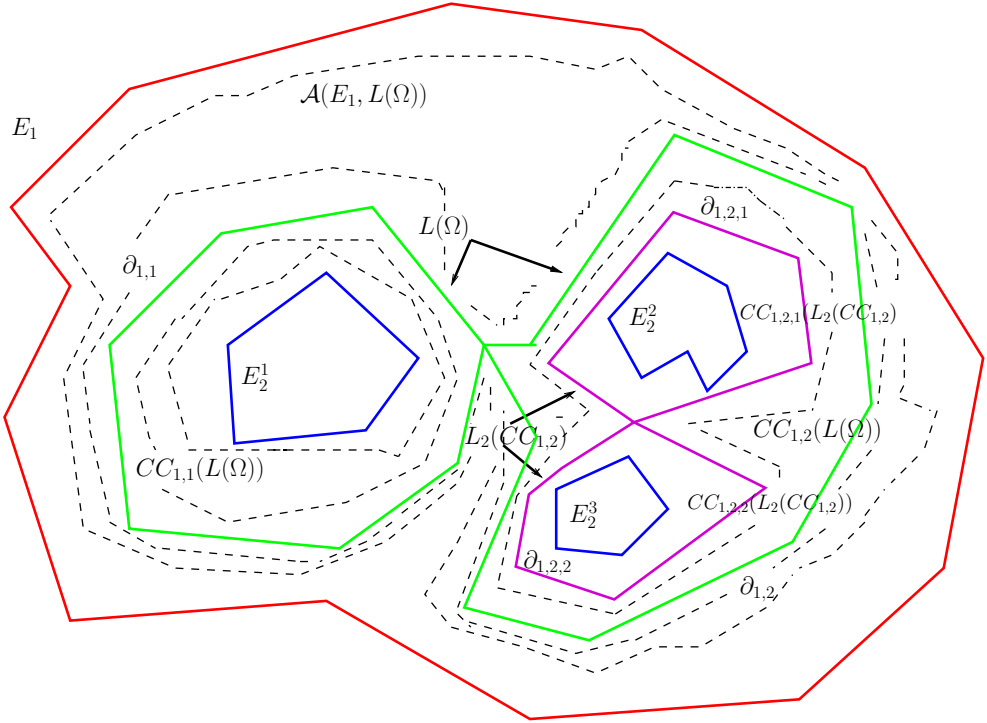}}
 \caption{An example of a splitting scheme.}
\label{figure:quad3}
\end{center}
\end{figure}

\smallskip
By definition, for each $j$ chosen as above, and each  $i=1,\ldots v$, the $g$-value on the boundary component 
\begin{equation}
\label{eq:procedure1}
\partial_{1,j,i}= \partial CC_{1,j,i}(L_j(CC_{1,j}))\cap L_j(CC_{1,j})
\end{equation}
is the constant which equals the $g$-value on $L_j(CC_{1,j})$. The other components of $\partial CC_{1,j,i}(L(\Omega))$ are kept at $g$-values equal to $0$. Hence, we now impose a (discrete) Dirichlet boundary value problem with the above values on each connected component in the list 
\begin{equation}
\label{eq:list1}
CC_{1,j}\setminus {\mathcal A}(\partial_{1,j},L_j(CC_{1,j})).
\end{equation}
On ${\mathcal A}(\partial_{1,j},
L_j(CC_{1,j}))$, the induced (discrete) Dirichlet boundary value problem is determined by the value of $g$ restricted to $\partial_{1,j}$, and the value of $g$ restricted to $L_j(CC_{1,j})$. Addition of vertices of type I and II, and modifications of conductance constants, will be applied in this step as in the previous one.

\medskip

It follows that in each step of the splitting scheme either a domain with fewer boundary components than the one that was split, or a singular annulus, or an annulus, is obtained. Hence, 
the connectivity level of each connected component after the split  is either the constant number two, the constant number three, or it decreases. Therefore, the splitting scheme will terminate once all the obtained components  have connectivity which equal to two or three, i.e., when the union of all the final generated lists is a list of lists, each containing only annuli and generalized singular annuli. 

\medskip

We now turn to the second part of the proof. Here, we will show that it is possible to reverse the splitting scheme, i.e., starting at the final lists generated in the splitting scheme up to the first one, $CC_1$, we will glue the pieces in a {\it geometric way}; that is, so that the lengths of glued boundary components are equal. It is in this step where the pair-flux length (Section~\ref{se:fgmetric}) will be used  (whenever the term ``length" appears).

\smallskip

By the structure of the lists obtained in the first part of the proof, it is sufficient to show how to 
\begin{enumerate}
\item
glue in a geometric way elements in a list, say  $CC_{1,j,\ldots,k_j}$,
that contains only annuli and a generalized singular annuli so as to form a ladder of singular pair of pants denoted by ${\mathcal S}_{1,j,\ldots,k_j}$, and

\item glue in a geometric way ${\mathcal S}_{1,j,\ldots,k_j}$ to the singular boundary component of
 the generalized singular annulus in the list from which  $CC_{1,j,\ldots,k_j}$ was formed.

\end{enumerate}

Note that, if all the elements in the list $CC_{1,j,\ldots,k_j}$ are annuli, we can apply case $(2)$, since we will map each annulus via 
Theorem~\ref{th:annulus} to a Euclidean cylinder (by first applying Theorem~\ref{th:annulus} and then Remark~\ref{re:to a cylinder}). 

\smallskip

In order to ease the notation, let us show steps $(1)$ and $(2)$ for the lists produced in the first part of the proof.

\smallskip

Fix a $j\in\{1,\ldots, p-1\}$, and for $i=1,\ldots, v-1$, we apply Theorem~\ref{th:annulus} to 

\begin{equation}
\label{eq:annend}
{\mathcal A}_i=CC_{1,j,i},
\end{equation} 
which by the assumption of step $(1)$ is an annulus. This yields a collection of concentric Euclidean annuli  $\{S_{1},\ldots,S_{v-1}\}$ where, in the notation of Theorem~\ref{th:annulus}, $S_i= S_{{\mathcal A}_i}$. Furthermore,
for each $i=1,\ldots, v-1$, we have
\begin{equation}
\label{eq:dimsi}
\{r_1^{i},r_2^{i}\}=  \{1, 2\pi\,\mbox{\rm Length}_{ (g,h_{i}) }(L( \partial_{1,j,i}   ))\}=\{1, 2\pi \exp\big( \frac{2\pi}{\mbox{\rm period}(g^{\ast}_i)}\,     g(\partial_{1,j,i})\big)\},
\end{equation}
where $g_i^{\ast}$ is the period (Definition~\ref{de:period}) 
of the conjugate function to $g$, the solution of the Dirichlet boundary value problem defined on $CC_{1,j,i}$.

\smallskip

 We now apply the map $F$ defined in Equation~(\ref{eq:annulustocylinder}) to obtain a corresponding sequence of Euclidean cylinders  $\{C_1,\ldots,C_{v-1}\}$. All of these have radii equal to the constant $1$, and their heights are given respectively by 
\begin{equation}
\label{eq:heights}
h_i=  \log{ (2\pi\,\mbox{\rm Length}_{(g,h_i)}(L( \partial_{1,j,i}   )))}=\log{ 2\pi}+  \frac{2\pi}{\mbox{\rm period}(g^{\ast}_i)}\,     g(\partial_{1,j,i}).
\end{equation}
     
\smallskip

 Recall that the last component in the list $CC_{1,j}$ is the generalized singular annulus 
${\mathcal A}(\partial_{1,j},L_j(CC_{1,j}))$. Let $h_{1,j}$ be the conjugate harmonic function to $g$, the solution of the Dirichlet boundary value problem induced on it (see the paragraph preceding Proposition~\ref{th:annulussing}). Then, following Remark~\ref{re:to a cylinder}, we now map it to a generalized Euclidean cylinder, $C_{{\mathcal A}(\partial_{1,j},L_j(CC_{1,j}))}$. 

\smallskip

Measuring lengths of the image of the boundary components of ${\mathcal A}(\partial_{1,j},L_j(CC_{1,j}))$ (using $g$ and $h_{1,j}$), which are in one to one correspondence with the singular boundary of this cylinder, yields the sequence of radii
\begin{equation}
\label{eq:almostthere}
 R_i= \frac{1}{2\pi}\frac{\mbox{\rm Length}_{g,h_j}(\partial_{1,j,i})}{ \mbox{\rm Length}_{g,h_j}(\partial_{1,j})} = \frac{1}{2\pi}\frac{1}{\mbox{\rm period}(g^{\ast}_{1,j})}    \frac{\exp\big(\frac{2 \pi}{\mbox{\rm period}(g^{\ast}_{i,j}) }    g( \partial_{1,j,i}   ))\big)}{\exp\big(\frac{2\pi}{\mbox{\rm period}(g^{\ast}_{1,j})} g(\partial_{1,j})\big)}     \int_{e\in\partial_{1,j,i}} |dh_{1,j}(e)|,
\end{equation}
for $i=1,\ldots,v-1$, where the expression on the righthand side is based on  Equation~(\ref{eq:length gradient metric1}), and Equation~(\ref{eq:levelg}). The sequence $R_i$, $i=1,\ldots, v-1$ comprises the sequence of lengths of the round circles in the singular boundary component of $C_{\mathcal A}(\partial_{1,j},L_j(CC_{1,j}))$.  
\smallskip

 Let 
\begin{equation}
\label{eq:conformalfactors}
f_i = f_{R_i}= R_i z
\end{equation}
be the conformal homeomorphism acting on the Euclidean cylinder $C_i=\pi\circ F\circ f(CC_{1,j,i})$, where $z$ is the standard complex parameter on $CC_{1,j,i}$. Hence, we may glue $f_i(C_i)$ along one of its boundary components to the  corresponding component in the singular boundary component of $C_{{\mathcal A}(\partial_{1,j},L_j(CC_{1,j}))}$, so that the length of the two boundaries are the same. 

\smallskip

This establishes step $(1)$. 

\smallskip
We will now show how to establish step $(2)$. 
 The completion of step $(1)$ yields  a ladder of singular pairs of pants which we will denote by ${\mathcal S}_{1,j}$. This ladder has $(m-1)$ components; one corresponds to $\partial_{1,j}$ and the others comprise one boundary component of each one of the cylinders $f_i\circ C_i$, $i=1,\ldots m-2$, after these cylinders are attached. Recall that  $\partial_{1,j}$ is the intersection of the generalized singular annulus ${\mathcal A}(E_1, L(\Omega))$,
 with the generalized singular annulus in the list $CC_{1,j}$, which is ${\mathcal A}(\partial_{1,j},L_j(CC_{1,j}))$. 

\smallskip
 
 Let $g$ be the solution of the induced Dirichlet boundary value problem on  ${\mathcal A}(E_1, L(\Omega))$, and let $h$ be the harmonic conjugate function to $g$ (see the paragraph preceding Proposition~\ref{th:annulussing}).
Let 
\begin{equation}
\label{eq:length in ladder}
\tau_j = \frac{l_j}{\mbox{\rm Length}_{g,h}( \partial_{i,j}   )   },
\end{equation}
where $l_j$ denotes the length of the boundary component that corresponds to $\partial_{1,j}$ in 
${\mathcal S}_{1,j}$, and $C_{{\mathcal A}(E_1, L(\Omega))}$ is the generalized Euclidean cylinder constructed for ${{\mathcal A}(E_1, L(\Omega))}$ (see Remark~\ref{re:to a cylinder}).  After applying a conformal expansion of magnitude $\tau_j$ to ${\mathcal S}_{1,j}$, it then may be glued along this boundary component to  $C_{{\mathcal A}(E_1, L(\Omega))}$ in such a way that the length of the corresponding circle in the round bouquet $\partial C_{{\mathcal A}(E_1, L(\Omega))}$ has the same length.  This establishes step $(2)$. 

\smallskip
By construction, it is clear that the pair-flux length of the boundary component of $S_{\Omega}$ that corresponds to $E_1$ is equal to
\begin{equation}
\label{eq:boundarylenght1}
2 \pi \exp{ (\frac{2 \pi}{ \mbox{\rm period} (g^{\ast})   } k)},
\end{equation}
where $h$ is the harmonic conjugate to $g$, the solution of the Dirichelt boundary value problem induced on ${\mathcal A}(E_1,L(\Omega))$ (recall from Definition~\ref{de:dbvp} that $k$ is the value of $g$ restricted to $E_1$).
The lengths of the remaining $(m-1)$ boundary components of $S_{\Omega}$, which correspond to the $m-1$ boundary components of $E_2$, are determined by the process described in the previous part of the proof. 

\smallskip

The length of a component in $S_{\Omega}$ which corresponds to $E_2^i$, $i\in \{1,\ldots,m-1\}$, measured with respect to the pair-flux metric,  is obtained by successively multiplying  a sequence of ratios of lengths. These ratios are uniquely determined as in Equation~(\ref{eq:almostthere}), and Equation~(\ref{eq:length in ladder}), and present the expansion factor needed in order to match the gluing of a (generalized) cylinder to the one which induced it in the splitting process. 

\smallskip

Cone angles are formed whenever more than two cylinders meet at a vertex; viewed in $\Omega$, this will occur whenever more than two circles in a generalized bouquet meet at a vertex. The computation of the cone angles is solely determined by $g$ and ${\mathcal T}$. This analysis first appeared in Theorem 0.4 in \cite{Her1}. Specifically, the cone angle $\phi(v)$ at a singular vertex $v$, which is the unique tangency point of $n+1$ Euclidean cylinders, satisfies
\begin{equation}
\label{eq:cone}
\phi(v)=2(n+1)\pi.
\end{equation}

 The proof of the theorem is thus complete, with $f$ defined to be the union of the individual maps constructed at each stage.

\end{proof}


\bigskip

\begin{thebibliography}{99}

\bibitem{Ah} L.~V.~Ahlfors, \emph{Conformal invariants-Topics in Geometric Function Theory}, McGraw-Hill Book Company, 1973.

\bibitem{An1} E.M.~Andreev, \emph{On convex polyhedra in Loba\u{c}evski\u{i} space}, Mathematicheskii Sbornik (N.S.) \textbf{81} (\textbf{123}) (1970), 445--478 (Russian); Mathematics of the USSR-Sbornik \textbf{10} (1970), 413--440 (English).

\bibitem{An2} E.M.~Andreev, \emph{On convex polyhedra of finite volume in Loba\u{c}evski\u{i} space},  
Mathematicheskii Sbornik (N.S.) \textbf{83} (\textbf{125}) (1970), 256--260 (Russian); Mathematics of the USSR-Sbornik \textbf{10} (1970), 255--259 (English).



\bibitem{Ba} T.~Banchoff, \emph{Critical points and curvature for embedded polyhedra}, J. Differential Geometry, \textbf{1} (1967), 245--256.

\bibitem{BeCaEn} E.~Bendito, A.~Carmona, A.M.~Encinas, \emph
{Solving boundary value problems on networks using equilibrium
measures}, J. of Func. Analysis, \textbf{171}  (2000), 155--176.

\bibitem{BeCaEn1} E.~Bendito, A.~Carmona, A.M.~Encinas, \emph
{Shortest Paths in Distance-regular Graphs}, Europ. J.
Combinatorics, \textbf{21} (2000), 153--166.

\bibitem{BeCaEn2} E.~Bendito, A.~Carmona, A.M.~Encinas, \emph
{Equilibrium measure, Poisson Kernel and Effective Resistance on Networks}, De Gruyter. 
Proceeding in Mathematics, (V. Kaimanovich, K. Schmidt, W. Woess ed.),  \textbf{174} (2003), 363--376.

\bibitem{BeCaEn3} E.~Bendito, A.~Carmona, A.M.~Encinas, \emph{Difference schemes on uniform grids performed by general discrete operators}, Applied 
Numerical Mathematics, \textbf{50} (2004), 343--370.

\bibitem{BeSch1} I.~Benjamini, O.~Schramm, \emph{Random walks and harmonic functions on infinite planar graphs using square tilings}, Ann. Probab. \textbf{24} (1996), 1219--1238.

\bibitem{BeSch2} I.~Benjamini, O.~Schramm, \emph{Harmonic functions on planar and almost
planar graphs and manifolds, via circle packings}, Invent. Math.
\textbf{126} (1996), 565--587.


\bibitem{BSST} R.L~Brooks, C.A.~Smith, A.B.~Stone and W.T.~Tutte, \emph{The dissection of squares into squares}, Duke Math. J. \textbf{7} (1940), 312--340.

\bibitem{Ca} J. W.~Cannon, \emph{The combinatorial Riemann mapping theorem},  Acta Math. \textbf {173} (1994), 155--234. 

\bibitem{CaFlPa} J.W.~Cannon, W.J.~Floyd and W.R.~ Parry,  \emph{Squaring rectangles: the finite Riemann mapping theorem},  Contemporary Mathematics,  Amer. Math. Soc., vol. \textbf{169}, Providence, 1994, 133--212.

\bibitem{ChSm} D.~Chelkak and S.~Smirnov, \emph{Discrete complex analysis on isoradial graphs},  
Adv. Math.  \textbf{228} (2011), 1590-1630. 

\bibitem{ChGrYa} F.R.~Chung, A.~Grigo\'{r}yan and S.T.~Yau, \emph
{Upper bounds for eigenvalues of the discrete and continuous
Laplace operators}, Adv. Math. \textbf{117} (1996), 165--178.

\bibitem{CoFrLe} R.~Courant, K.~Friedrichs and H.~Lewy, \emph{\"Uber die partiellen Differenzengleichungen der mathematischen Physic}, Math. Ann {\textbf 100} (1928), 32 -- 74.


\bibitem{D} M.~Dehn, \emph{Zerlegung ovn Rechtecke in Rechtecken}, Mathematische Annalen, \textbf {57}, (1903), 144-167.

\bibitem{Du} R.~Duffin, \emph{The extremal length of a network}, J. Math. Anal. Appl. {\textbf 5} (1962), 200--215.


\bibitem{Fu} B.~Fuglede, \emph{On the theory of potentials in
locally compact spaces}, Acta. Math.  \textbf{103} (1960), 139--215.

\bibitem{Gl} D.~Glickenstein.\emph{Discrete conformal variations and scalar curvature on piecewise flat two- and three-dimensional manifolds}, J. Differential Geom. {\bf 87} (2011),  201--237.

\bibitem{GuLuYa} G.X.~David, F.~Luo and S.H.~Yau, \emph{Recent advances in computational conformal geometry}, Commun. Inf. Syst. {\textbf 9} (2009), 163--195.


\bibitem{GuZeLuYa}  G.X.~David, F.~Luo, Z.~Wei and S.H.~Yau, \emph{Numerical computation of surface conformal mappings}, Comput. Methods Funct. Theory {\textbf 11}  (2011), 747--787.



\bibitem{Her} S.~Hersonsky, \emph{Energy and length in a topological planar quadrilateral}, Euro. Jour.  of Combinatorics  \textbf{29}  (2008), 208-217.                                           


\bibitem{Her1} S.~Hersonsky, \emph{Boundary Value Problems on Planar Graphs and Flat Surfaces with Integer Cone singularities I; The Dirichlet problem}, J. Reine Angew. Math. \textbf{670}, (2012), 65--92.


\bibitem{Her2} S.~Hersonsky, \emph{Boundary Value Problems on Planar Graphs and Flat Surfaces with Integer Cone singularities II; Dirichlet-Neumann problem}, Differential Geometry and its Applications \textbf{29} (2011), 329--347.




\bibitem{Her3} S.~Hersonsky, \emph{Combinatorial Harmonic Maps and Convergence to Conformal Maps, II: Convergence and Moduli.}, in preparation.


\bibitem{Hi} A.N.~Hirani, \emph{Discrete exterior calculus}, Dissertation (Ph.D.), California Institute of Technology, 	http://resolver.caltech.edu/CaltechETD:etd-05202003-095403.

\bibitem{Ke} R.~Kenyon, \emph{Tilings and discrete Dirichlet problems}, Israel J. Math. \textbf{105} (1998), 61--84.


\bibitem{Koe0} P.~Koebe, \emph{Uber die Unifirmiseriung beliegiger analytischer Kurven III}, Nachrichten Gesellschaft f\"ur  Wisseenschaften in G\"ottingen (1908), 337--358.


\bibitem{Koe} P.~Koebe, \emph{Kontaktprobleme der Konformen Abbildung}, 
Ber. S\"{a}chs. Akad. Wiss. Leipzig, Math.-Phys. Kl. \textbf{88} (1936) 141--164.



\bibitem{LaVe} F.~Lazarus and A.~Verroust, \emph{Level Set Diagrams of Polyhedral Objects}, ACM Symposium on Solid and Physical Modeling, Ann Arbor, Michigan, (1999), 130--140.


\bibitem{Lu} F.~Luo, \emph{Variational principles on triangulated surfaces}, Handbook of geometric analysis. No. 1, 259Ð276, 
Adv. Lect. Math. (ALM), 7, Int. Press, Somerville, MA, 2008.


\bibitem{Mer} C.~Mercat, \emph{Discrete Riemann surfaces}, Handbook of Teichm\"uller theory. Vol. I, 541Ð575, IRMA Lect. Math. Theor. Phys., 11, Eur. Math. Soc., Z\"urich, 2007.


\bibitem{Po} K. Polthier, \emph{Computational aspects of discrete minimal surfaces},  Global theory of minimal surfaces, 65Ð111, Clay Math. Proc., 2, Amer. Math. Soc., Providence, RI, 2005.


\bibitem{RoSu} B.~Rodin and D.~Sullivan, \emph{The convergence of circle packing to the Riemann mapping}, Jour. Differential Geometry {\textbf 26} (1987), 349--360. 

\bibitem{SaStBr} C.T.~Sass, K.~Stephenson and W.G.~Brock, \emph{Circle packings on conformal and affine tori}, Computational algebraic and analytic geometry, 211--220, Contemp. Math., 572, Amer. Math. Soc., Providence, RI, 2012.


\bibitem{Sch} O.~Schramm, \emph{Square tilings with prescribed combinatorics}, Israel Jour. of Math. \textbf{84} (1993), 97--118.



\bibitem{So} P.M.~Soardi, \emph{Potential theory on infinite networks}, Lecture Notes in Mathematics, \textbf{1590}, Springer-Verlag Berlin Heidelberg 1994. 


\bibitem{Sp} G.~Springer, \emph{Introduction to Riemann surfaces}, Addison-Wesley Publishing Company, Inc., Reading, Mass. 1957.

\bibitem{Steph} K.~Stephenson, \emph{A Probabilistic Proof of Thurston's Conjecture on Circle Packings}, Rendiconti del Seminario Mate. e Fisico di Milano, LXVI (1996), 201--291.


\bibitem{Th1} W.P.~Thurston, \emph{The finite Riemann mapping theorem}, invited address, International Symposium in Celebration of the Proof of the Bieberbach Conjecture, Purdue University, 1985.

\bibitem{Th2} W.P.~Thurston, \emph{The Geometry and Topology of $3$-manifolds}, Princeton University Notes, Princeton, New Jersey, 1982.




\end{thebibliography}
\end{document}